\RequirePackage{fix-cm}
\documentclass[smallextended]{svjour3}       
\smartqed  
\usepackage{graphicx}
\usepackage{yuyuan}
\usepackage{hyperref}
\usepackage{algorithm, algpseudocode, multirow}
\newcommand\qedsymbol{\begin{flushright}$\blacksquare$\end{flushright}}
\usepackage[a4paper, margin=1.2in]{geometry}

\newcommand{\xp}{x^+}
\newcommand{\xpp}{x^{++}}

\newcommand{\xu}{\overline{x}}

\newcommand{\cmo}[1]{{#1}}
\newcommand{\mt}{accumulative regularization}

\newtheorem{assumption}{Assumption}
\numberwithin{lemma}{section}
\numberwithin{theorem}{section}
\numberwithin{proposition}{section}
\numberwithin{algorithm}{section}
\numberwithin{assumption}{section}
\numberwithin{equation}{section}
\numberwithin{remark}{section}

\begin{document}

\title{Optimal and parameter-free gradient minimization methods for
	convex and nonconvex optimization
	\thanks{This work is partially supported by AFOSR grant FA9550-22-1-0447 and ONR grant N000142412621.}
}


\author{Guanghui Lan         \and Yuyuan Ouyang \and Zhe Zhang
}


\institute{G. Lan \at
	H. Milton Stewart School of Industrial and Systems Engineering, Georgia Institute of Technology, Atlanta, GA \\
	\email{george.lan@isye.gatech.edu}
	\and
	Y. Ouyang \at
        School of Mathematical and Statistic Sciences, Clemson University, Clemson, SC  
     \\
        \email{yuyuano@clemson.edu}
           \and
    Zhe Zhang \at
		School of Industrial Engineering, Purdue University, West Lafayette, IN 
		\\
		\email{zhan5111@purdue.edu}
}

\date{Received: date / Accepted: date}

\maketitle

\begin{abstract}
We propose novel optimal and parameter-free algorithms for computing an approximate solution with small (projected) gradient norm. Specifically, for computing an approximate solution such that the norm of its (projected) gradient does not exceed $\varepsilon$, we obtain the following results:
a) for the convex case, the total number of gradient evaluations is bounded by $\cO(1)\sqrt{L\|x_0 - x^*\|/\varepsilon}$, where $L$ is the Lipschitz constant of the gradient and $x^*$ is any optimal solution; b) for the strongly convex case, the total number of gradient evaluations is bounded by $\cO(1)\sqrt{L/\mu}\log(\|\nabla f(x_0)\|/\epsilon)$, where $\mu$ is the strong convexity modulus; and c) for the nonconvex case, the total number of gradient evaluations is bounded by $\cO(1)\sqrt{Ll}(f(x_0) - f(x^*))/\varepsilon^2$, where $l$ is the lower curvature constant. Our complexity results match the lower complexity bounds of  the convex and strongly cases, and achieve the above best-known complexity bound for the nonconvex case for the first time in the literature. \cmo{Our results can also be extended to problems with constraints and composite objectives.}
Moreover, 
for all the convex, strongly convex, and nonconvex cases, we propose parameter-free algorithms that do not require the input of any problem parameters \cmo{ or the convexity status of the problem}. To the best of our knowledge, there do not exist such parameter-free methods before especially for the strongly convex and nonconvex cases. Since most regularity conditions (e.g., strong convexity and lower curvature) are imposed over a global scope, the corresponding problem parameters  are notoriously difficult to estimate. However, 
gradient norm minimization equips us with
a convenient tool to monitor the progress of
algorithms and thus the ability to estimate
such parameters in-situ.
\end{abstract}

\section{Introduction}

The basic problem of interest in this paper is to compute an
approximate stationary point $\hat x$ to  
the unconstrained optimization problem
\begin{equation} \label{eq:basic_problem}
	f^* = \min_{x\in\R^n}f(x),
\end{equation}
where $f$ is $L$-smooth over $\R^n$, meaning that $f$ is differentiable and its gradient $\nabla f$ is Lipschitz continuous, i.e., $\|\nabla f(x) - \nabla f(y)\|
\le L \|x - y\|$ for any $x, y \in \R^n$. \cmo{ Here and throughout this paper, the norm $\|\cdot\|$ measured over points $x\in \R^n$ and gradients $\nabla f(x)\in\R^n$ are all Euclidean norm.}
Moreover, we consider three important types of $f$ as follows: (a) $f$ is convex, i.e.,
$f(x) \ge f(y) +\langle \nabla f(y), x-y\rangle$ for any $x, y \in \R^n$; (b) $f$ is $\mu$-strongly convex, i.e., $f(x) \ge f(y) +\langle \nabla f(y), x-y\rangle + \mu \|x-y\|^2/2$ $\forall x, y \in \R^n$ for some $\mu \in (0, L)$;
and (c) $f$ is nonconvex with lower curvature $l$, i.e., $f(x)$ is nonconvex but $f(x) + l \|x\|^2/2$ becomes convex for some $l \in (0, L]$.
We are interested in the following two fundamental
questions: a) how many number of 
gradient evaluations (a.k.a, gradient complexity)
is needed to attain $\|\nabla f(\hat x)\|\le \varepsilon$
for different types of objective function $f$,
and b) how much input information is needed
for the algorithms to find such a solution. 
Note that our discussion will also be generalized to  the constrained case where 
$\nabla f(\hat x)$ is replaced by the projected gradient at $\hat x$. Our understanding for these questions, however, appears to be quite inadequate for both unconstrained and constrained cases.

Let us start with a review on the gradient complexity.
Consider the case when $f$ is convex.
Classic optimal methods that can achieve the best possible
gradient complexity for convex optimization~\cite{nemirovski1983problem,nesterov2004introductory}
were designed for the termination criterion $f(\hat x)-f(x^*)\le\varepsilon$ instead of
the gradient norm $\|\nabla f(\hat x)\| \le \varepsilon$.
It is well-known that the former criterion is difficult to check since $f^*$ is unknown while the latter one is easily verifiable.
Moreover, the latter criterion is considered to be stronger due to
$\|\nabla f(\hat x)\|^2/(2L) \le f(\hat x) - f^* \le \|\nabla f(\hat x)\| \|\hat x - x^*\|$.
However,
for a long period of time, there only exist suboptimal
methods for gradient norm minimization~(see, e.g., \cite{nesterov2012make,ghadimi2016accelerated,nesterov2018lectures} and the references within).
The best earlier upper bounds on gradient complexity \cite{nesterov2018lectures} (see also \cite{allen2018make,foster2019complexity,lu2023accelerated}) is worse than the lower complexity bound $\cO(\sqrt{L\|x_0-x^*\|/\varepsilon})$~\cite{nemirovski1983problem,nemirovski1992information,carmon2020lower} by a logarithmic factor. 
Until recently it has been discovered in a series of  interesting works~\cite{kim2021optimizing,nesterov2021primal,diakonikolas2022potential,lee2021geometric} that
the aforementioned lower complexity bound 
for gradient norm minimization is actually achievable.
More specifically, 
Kim et. al.~\cite{kim2021optimizing} introduces an optimized gradient
method for gradient minimization (OGM-G)
that can find a solution $\hat x$ such that $\|\nabla f(\hat x)\|\le \varepsilon$ within $\cO(\sqrt{L(f(x_0) - f(x^*))}/\varepsilon)$ gradient evaluations.
It was later pointed out in \cite{nesterov2021primal} that by running $N$ iterations of any  optimal gradient method for minimizing objective function value, followed with $N$ iterations of OGM-G, one can compute $\hat x$ with small gradient norm within $\cO(\sqrt{L\|x_0-x^*\|/\varepsilon})$ gradient evaluations. Note that the original OGM-G method was designed by utilizing the framework of performance enhancement program (PEP)~\cite{drori2014performance} to
search for a set of algorithmic parameters in gradient methods through the empirical solutions of a nonconvex semidefinite programming. The authors in \cite{diakonikolas2022potential,lee2021geometric} suggest to analyze OGM-G using different
potential functions, and 
the work \cite{lee2021geometric} also generalizes OGM-G to solve constrained problems. 

In spite of much progresses on optimal gradient  minimization methods for convex optimization, there remain a few important unresolved issues. First, OGM-G requires one to supply the total number of iterations $N$ in advance. As a consequence, we do not actually utilize the easily verifiable condition $\|\nabla f(\hat x)\| \le \epsilon$ to terminate the algorithm. 
Second, to achieve the optimal complexity we need to run two different algorithms consecutively, where
the first algorithm computes an approximate solution $\tilde x$ such that $f(\tilde x) - f(x^*)\le\varepsilon$ and the second one uses $\tilde x$ as the initial point to compute an approximate solution $\hat x$ with $\|\nabla f(\hat x)\|\le\varepsilon$. This two-step approach, although technically sound, appears to lack intuitive interpretation. 
Third, the aforementioned studies focus on the general convex case only, and the best possible complexity bound remains unknown when $f$ is $\mu$-strongly convex. 
Specifically,
by using the accelerated gradient method \cite{nesterov2018lectures} 
we can guarantee $\|\nabla f(\hat x)\|\le\varepsilon$ within $\cO(1)\sqrt{L/\mu}\log(L/(\mu \varepsilon))$ gradient evaluations. It is unclear whether one could remove the logarithmic dependence on the condition number $L/\mu$ to
obtain an optimal $\cO(1) \sqrt{L/\mu}\log(1/\varepsilon)$
complexity bound.
Fourth, 
it remains unclear what is the best possible complexity bound for nonconvex problems. 
Some results have been developed in the literature (see, e.g., \cite{ghadimi2016accelerated,lan2019accelerated,liang2023proximal} and the references therein) and the best-known complexity is given by $\cO(1)\sqrt{L l}(f(x_0) - f(x^*)) (1+\log (L/l))/\varepsilon^2$~\cite{lan2019accelerated}.
Note that if $l = L$ then this bound is known to be optimal~\cite{carmon2020lower}. In fact, the optimal bound in the latter special case can be simply achieved by the gradient descent method. However, it is unclear if the above complexity can be improved for the more general case when $l \in (0, L]$. 

It should be noted that all the works listed in the above discussion require some knowledge of problem parameters (e.g., $L$, $\mu$ and $l$).
These problem parameters are usually difficult to estimate accurately, and a conservative estimation can dramatically slow down the algorithms.
As such, algorithms that can achieve the same complexity bounds without requiring the input of these parameters has attracted much attention recently. Existing studies on these parameter-free methods  have focused on computing an approximate solution $\hat x$ satisfying $f(\hat x) - f^*\le\varepsilon$ (see, e.g., \cite{lan2015bundle,nesterov2015universal,ChenLanOuyangZhang2019bundle,lan2023simple,renegar2022simple,nesterov2013gradient,lin2014adaptive,malitsky2023adaptive,lu2023accelerated} and references therein). In particular, for convex and $L$-smooth functions, parameter-free algorithms with optimal gradient evaluation complexity are often called uniformly (or universally) optimal methods.
These types of methods have been developed first based on accelerated bundle-level method \cite{lan2015bundle} and later on line search procedures for accelerated gradient methods \cite{nesterov2015universal}. Recently, it is shown in \cite{lan2023simple} that line search is not required for uniformly optimal gradient methods.
Note that since most regularity conditions (e.g., strong convexity and lower curvature) are imposed over a global scope, the corresponding problem parameters $\mu$ and $l$ are notoriously difficult to estimate. To the best of our knowledge,
there does not exist any algorithm that does not require the input of $\mu$ and $l$, but can still maintain the same gradient complexity for strongly convex and nonconvex problems.
Since our goal is to find an $\hat x$ s.t. $\|\nabla f(\hat x)\|\le\varepsilon$, there does not exist any parameter-free and optimal algorithm even for the convex case, not to mention the more challenging strongly convex and nonconvex cases.

In order to address these issues mentioned above, in this paper we propose several novel gradient minimization algorithms that can achieve optimal or best-known complexity bounds to solve convex, strongly convex, and nonconvex problems. In fact, our new complexity results for strongly convex and nonconvex problems have never been attained before in the literature. We further show that the proposed algorithms can achieve these complexity bounds without requring the input of any problem parameters. Moreover, our algorithms can be easily extended to problems with constrained feasible sets, and we demonstrate such possibility on convex problems with convex or (nested) composite objective function over a constrained convex feasible set. The contributions of this paper is described in detail below.

First, for unconstrained and convex problems, we develop an accumulative regularization algorithm for gradient minimization,
obtained by introducing regularization in an accumulative fashion
into the classic proximal-point method. We show that this algorithm can compute an approximate solution $\hat x$ such that $\|\nabla f(\hat x)\|\le \varepsilon$ within at most $\cO(1)\sqrt{L\|x_0 - x^*\|/\varepsilon}$ gradient evaluations. While such optimal complexity has been achieved before in \cite{kim2021optimizing,nesterov2021primal}, our proximal point type method appears to be simpler and more intuitive. In contrast to the previous work
that requires a two-step implementation using two separate algorithms, our proposed algorithm is guided by a single accumulative regularization strategy throughout all the iterations. 
\cmo{ Specifically, our regularization parameter depends on the accuracy threshold $\varepsilon$ at the beginning and increases exponentially as the algorithm proceeds forward.
}
It is worth noting that a regularization strategy similar to ours has been studied before in \cite{allen2018make}. However, the algorithm in \cite{allen2018make} achieves only a suboptimal gradient complexity bound that involves an extra multiplicative logarithmic factor.

Second, for constrained and convex problems, we show that our accumulative regularization algorithm can  compute an approximate solution $\hat x$ whose project gradient norm does not exceed  $\varepsilon$ with the same complexity bound as for the unconstrained case. We also show that our algorithm can be further adapted to problems with composite and/or nested composite objective function. 

Third, we design a parameter-free algorithm for computing an approximate solution with small gradient for unconstrained convex optimization with the same optimal gradient evaluation complexity. To the best of knowledge, no parameter-free algorithm has yet appeared in the literature of unconstrained gradient minimization.

Fourth, for $\mu$-strongly convex problems, we
show that by properly restarting the accumulative regularization algorithm method, one can compute an approximate solution $\hat x$
s.t. $\|\nabla f(\hat x)\| \le \varepsilon$ within at most 
$\cO(1)\sqrt{L/\mu}\log(\|\nabla f(x_0)\|/\epsilon)$ gradient evaluations. This bound removes an extra logarithmic dependence on
the condition number from existing results and matches the lower bound in the literature \cite{nemirovski1983problem,nemirovski1992information,carmon2020lower}.
Moreover, we
develop a parameter-free strongly convex accumulative regularization (SCAR) algorithm that maintains the same gradient complexity without the input of any problem parameters. To the best of our knowledge, SCAR is the first parameter-free algorithm with optimal complexity for minimizing strongly convex functions. 

Fifth, for nonconvex problems with $L$-smooth and $l$-lower curvature objective function, we develop a nonconvex acceleration method by calling the SCAR algorithm iteratively, and show that this method is able to compute an approximate stationary point $\hat x$ s.t. $\|\nabla f(\hat x)\| \le \epsilon$ within at most $\cO(1)\sqrt{Ll}(f(x_0) - f(x^*))/\varepsilon^2$ gradient evaluations. This is the first time such complexity is achieved in the literature for nonconvex problems. 

Sixth, we develop a variant of the aforementioned SCAR algorithm, namely SCAR-PM, that can handle plausible strong convexity information and even detect the existence of nonconvexity. By calling SCAR-PM iteratively, we propose a parameter-free nonconvex acceleration through strongly convex accumulative regularization (NASCAR) algorithm that can maintain the same gradient evaluation complexity without any problem parameter information. To the best of our knowledge, NASCAR is the first parameter-free algorithm with $\cO(1)\sqrt{Ll}(f(x_0) - f(x^*))/\varepsilon^2$ gradient complexity for solving nonconvex problems.

\subsection{Organization of the paper}
We start with the discussion on gradient minimization for unconstrained convex problems in Section~\ref{sec:unconstrained}. In this section, we describe the accumulative regularization strategy and perform the complexity analysis under the assumption that problem parameters like the Lipschitz constant $L$ and distance to optimal solutions are known. In Section~\ref{sec:constrained}, we show that our proposed accumulative regularization strategy can be extended to solve problems with constrained feasible sets, composite and/or nested optimization structures. In Section~\ref{sec:pf}, we propose a parameter-free accumulative regularization (AR) algorithm that can be called to solve unconstrained convex problems without the knowledge of any problem parameter. In Section~\ref{sec:sc}, we show that by restarting the AR algorithms, we can design a parameter-free strongly convex accumulative regularization (SCAR) algorithm that solves the strongly convex and unconstrained problems with optimal complexity. In Section~\ref{sec:NCAR}, we show that the SCAR algorithm can be called iteratively to solve nonconvex problems with best known so-far complexity bounds, and in Section~\ref{sec:NCAR_PF} we propose a parameter-free nonconvex acceleration through strongly convex accumulative regularization (NASCAR) algorithm to achieve the same gradient complexity. Some concluding remarks are made in Section~\ref{sec:conclusion}.

\section{Minimizing gradient for unconstrained convex problems}
\label{sec:unconstrained}
We start our discussion with unconstrained convex optimization problem
$
\min_{x\in\R^n}f(x)
$
with $f$ being $L$-smooth. 
Our goal is to compute an $\varepsilon$-approximate solution $x$ satisfying $\|\nabla f(x)\|\le\varepsilon$. In this section, we will propose an algorithm for computing such an approximate solution that requires at most $\cO(\sqrt{L\|x_0 - x^*\|/\varepsilon})$ gradient evaluations. 

\subsection{An accumulative regularization method}
\label{sec:proposed_unconstrained}
The design of our algorithm relies on
the so-called \emph{\mt}, where 
a search point $x_s$ is generated as an approximate solution to 
the proximal mapping
$
\min_{x\in\R^n}f(x) + ({\sigma_s}/{2})\|x - \xu_s\|^2.
$
Comparing with the original problem, the above proximal mapping  has an additional regularization term that enforces our solution to be in close proximity to a point $\xu_s$. Throughout this paper, we will call $\sigma_s$ the \emph{regularization parameter} and $\xu_s$ the \emph{prox-center}.
Observe that our algorithm is closely related to the classic proximal point algorithm, where the prox-center 
is usually set to $x_{s-1}$, and the regularization parameter is often a constant.
We use the terminology ``accumulative'' to indicate that the prox-center $\xu_s$ is a convex combination of previous approximate solutions $x_0, \ldots, x_{s-1}$ and that the regularization parameter $\sigma_s$ is increasing. Our accumulative choice of the prox-center $\xu_s$ shares some foundamental principles with the second proximal point algorithm proposed in \cite{guler1992new}, in the sense that it is a linear combination of previous search points $x_s$. The difference is that we are using a convex combination and our result is for gradient minimization, while \cite{guler1992new} focuses on objective value minimization.
Our proposed algorithm is formally described in Algorithm~\ref{alg:proposed_unconstrained}.

\begin{algorithm}[h]
	\caption{\label{alg:proposed_unconstrained} An accumulative regularization method for gradient minimization}
	\begin{algorithmic}
		\Require Total number of subproblems $S$, strictly increasing regularization parameters $\{\sigma_s\}_{s=0}^{S}$ with $\sigma_0=0$, and initial point $x_0\in\R^n$. \State Set initial prox-center to $\xu_0 := x_0$.
		\For {$s=1,\ldots, S$}
		\State Set
		\begin{align}
			\label{eq:xus_unconstrained}
			\xu_s = (1-\gamma_s) \xu_{s-1} + \gamma_s x_{s-1}\text{ with }\gamma_s =1 - \sigma_{s-1}/\sigma_s.
		\end{align}
		\State Compute an approximate solution $x_s$ of the proximal mapping subproblem
		\begin{align}
			\label{eq:subproblem_unconstrained}
			x_s^*:=\argmin_{x\in \R^n} \{f_s(x):=f(x) + \frac{\sigma_s}{2}\|x - \xu_s\|^2\}
		\end{align}
		by running subroutine $\cA(f, \sigma_s, \xu_s, x_{s-1}, N_s)$ with $N_s$ gradient evaluations of $\nabla f$ and initial point $x_{s-1}$.
		\EndFor
		\State Output $x_S$.
	\end{algorithmic}
\end{algorithm}

A few remarks are in place for our proposed accumulative regularization algorithm in Algorithm~\ref{alg:proposed_unconstrained}. 
First, at the $s$-th iteration we will call a subroutine $\cA$ to minimize a smooth and strongly convex function $f_s$, starting from initial point $x_{s-1}$. 
Second, by definition \eqref{eq:xus_unconstrained}, the prox-center $\xu_s$ is a convex combination of previous approximate solutions $x_0,\ldots,x_{s-1}$. 
Third, while linear convergence is usually expected for solving smooth and strongly convex problems, an algorithm $\cA$ with a sublinear rate of convergence is sufficient for our analysis in the sequel. Indeed, throughout our analysis we will make the following assumption regarding the convergence of algorithm $\cA$.

\vgap

\begin{assumption}
	\label{assum:A_unconstrained}
	Subroutine $\cA$ used to compute $x_s =\cA(f, \sigma_s, \xu_s, x_{s-1}, N_s)$ exhibits the following performance guarantees after $N_s$ gradient evaluations of $\nabla f$:
	\begin{align}
		\label{eq:A_cond_unconstrained}
		f_s(x_s) - f_s(x_s^*) \le ({c_{\cA}L}/{N_s^2})\|x_{s-1} - x_s^*\|^2.
	\end{align} 
	Here $c_\cA$ is a universal constant that depends on subroutine $\cA$. 
\end{assumption}

\vgap

The above assumption on sublinear convergence of $\cA$ is satisfied by many algorithms, e.g., the accelerated gradient method \cite{nesterov2018lectures} with $c_\cA = 2$.

\subsection{Convergence analysis}
We outline some basic ideas of our convergence analysis before diving into technical details. 
The main concept behind our analysis of Algorithm~\ref{alg:proposed_unconstrained} is to
bound the gradient norm $\|\nabla f(x_s)\|$
by 
using the distances between $x_i$ and $x_i^*$, $i = 0, \ldots, s$. Then we will study the convergence properties  
in terms of these distances thanks to the strong convexity and accumulative regularization terms in \eqref{eq:subproblem_unconstrained}.
Specifically, 
noting the $L$-Lipschitz continuity of $\nabla f(x)$ and using the optimality condition $\nabla f_s(x_s^*)=\nabla f(x_s^*) + \sigma_s(x_s^* - \xu_s) = 0$ at $x_s^*$,
we have
\begin{align}
	\label{eq:xs_norm_bound_unconstrained}
	\begin{aligned}
		& \|\nabla f(x_s)\| = \|\nabla f(x_s) - \nabla f(x_s^*) - \sigma_s(x_s^* - \xu_s)\| 
		\\
		\le \ & \|\nabla f(x_s) - \nabla f(x_s^*)\| + \sigma_s\|x_s^* - \xu_s\| 
		\\
		\le\ &  L\|x_s - x_s^*\| + \sigma_s\|x_s^* - \xu_s\|.
	\end{aligned}
\end{align}
We will show later that with the increasing regularization parameter $\sigma_s$ and convex combination design of prox-center in \eqref{eq:xus_unconstrained}, the distance $\sigma_s\|x_s^* - \xu_s\|$ can be further bounded:
\begin{align}
	\label{eq:claim_on_distance_unconstrained}
	\sigma_s\|x_s^* - \xu_s\| \le \sigma_1\|x^* - x_0\| + \sum_{i=2}^s(\sigma_{i-1} + \sigma_i)\|x_{i-1}^* - x_{i-1}\|.
\end{align}
Therefore, $\|\nabla f(x_s)\|$ can be controled by
\begin{align}
	\label{eq:gradS_estimate_unconstrained}
	\|\nabla f(x_s)\| \le\ &  L\|x_s - x_s^*\| + \sigma_1\|x^* - x_0\| + \sum_{i=2}^s(\sigma_{i-1} + \sigma_i)\|x_{i-1}^* - x_{i-1}\|.
\end{align}
Note that the right-hand-side of the above estimate is the summation of terms of form $\|x_i - x_i^*\|$, i.e., the distance from an approximate solution $x_i$ to the optimal solution $x_i^*$ of the $i$-th strongly convex subproblem.
By studying the convergence properties of subroutine $\cA$ for minimizing smooth and strongly convex functions in subproblem \eqref{eq:subproblem_unconstrained}, we will then bound the total number of gradient evaluations needed to reduce the right-hand-side of the above estimate to the accuracy $\varepsilon$. 

We are now ready to describe the details for proving our claim \eqref{eq:claim_on_distance_unconstrained} and estimating the number of gradient evaluations required to solve subproblems \eqref{eq:subproblem_unconstrained}. For convenience, we allow that the definition   \eqref{eq:subproblem_unconstrained} applies for $s=0$, and in such case we have $\sigma_0 = 0$ and hence $x_0^*=x^*$ is an optimal solution to the original problem. The proof of our claim \eqref{eq:claim_on_distance_unconstrained} is described in the following proposition. Note that the result \eqref{eq:dist_xu_xstarsig_unconstrained} below is exactly our claim
\eqref{eq:claim_on_distance_unconstrained} since $\sigma_0 = 0$.

\vgap

\begin{proposition}
	\label{pro:xstar_dist_shrink_unconstrained}
	In  Algorithm~\ref{alg:proposed_unconstrained} we have for all $s\ge 1$ 
	\begin{align}
		\label{eq:dist_subproblem_unconstrained}
		& \|x_{s-1} - x_s^*\|\le \|x_{s-1} - x_{s-1}^*\|, 
		\\
		\label{eq:dist_xu_xstarsig_unconstrained}
		&\sigma_s\|\xu_s - x_s^*\| 
		\le \sum_{i=1}^{s}(\sigma_{i-1} + \sigma_{i})\|x_{i-1}^* - x_{i-1}\|
		.
	\end{align}
\end{proposition}
\begin{proof}
	With the selection of $\gamma_s = 1 - \sigma_{s-1}/\sigma_s$ in the definition of prox-center $\xu_s$ in   \eqref{eq:xus_unconstrained}, the subproblem \eqref{eq:subproblem_unconstrained} is equivalent to
	\begin{align*}
		x_s^* = \argmin_{x\in \R^n} f(x) + \dfrac{\sigma_{s-1}}{2}\|x - \xu_{s-1}\|^2 + \dfrac{\sigma_s - \sigma_{s-1}}{2}\|x - x_{s-1}\|^2,\ \forall s \ge 1.
	\end{align*}
	Noting the optimality of $x_{s-1}^*$ in \eqref{eq:subproblem_unconstrained}  and $x_s^*$ in the above subproblem we have
	\begin{align*}
		& f(x_{s-1}^*) + \dfrac{\sigma_{s-1}}{2}\|x_{s-1}^* - \xu_{s-1}\|^2 + \dfrac{\sigma_s - \sigma_{s-1}}{2}\|x_s^* - x_{s-1}\|^2
		\\
		\le\ & f(x_s^*) + \dfrac{\sigma_{s-1}}{2}\|x_s^* - \xu_{s-1}\|^2 + \dfrac{\sigma_s - \sigma_{s-1}}{2}\|x_s^* - x_{s-1}\|^2 
		\\
		\le\ & f(x_{s-1}^*) + \dfrac{\sigma_{s-1}}{2}\|x_{s-1}^* - \xu_{s-1}\|^2 + \dfrac{\sigma_s - \sigma_{s-1}}{2}\|x_{s-1}^* - x_{s-1}\|^2.
	\end{align*}
	The conclusion in \eqref{eq:dist_subproblem_unconstrained} then follows from the above relation and the assumption $\sigma_s>\sigma_{s-1}$ required by Algorithm~\ref{alg:proposed_unconstrained}.
	
	We move to prove the next result \eqref{eq:dist_xu_xstarsig_unconstrained}. Denoting $\alpha_s:=\sigma_s - \sigma_{s-1}$ for all $s$ and noting that $\gamma_s = \alpha_s/\sigma_s$ in \eqref{eq:xus_unconstrained}, we can rewrite the definition of $\xu_s$ in \eqref{eq:xus_unconstrained} to
	$
	\sigma_s \xu_s = \sigma_{s-1}\xu_{s-1} + \alpha_s x_{s-1}.
	$
	Recalling that $\sigma_0=0$, the above recursive relation yields
	$
	\sigma_s\xu_s = \sum_{i=1}^{s}\alpha_ix_{i-1}.
	$
	In other words, $\xu_s$ is a convex combination of $x_0,\ldots, x_{s-1}$ with weights $\alpha_i/\sigma_s$  since $\sigma_s = \sum_{i=1}^{s}\alpha_i$. Thus 
	\begin{align*}
		\sigma_s(\xu_s - x_s^*) =\ & \sum_{i=1}^{s}\alpha_i(x_{i-1} - x_s^*) 
		\\
		=\ & \alpha_s(x_{s-1} - x_s^*) + \sum_{i=1}^{s-1}\alpha_i(x_{i-1} - x_{s-1}^*)  + \left(\sum_{i=1}^{s-1}\alpha_i\right)(x_{s-1}^* - x_s^*)
		\\
		=\ & \alpha_s(x_{s-1} - x_s^*) + \sigma_{s-1}(\xu_{s-1} - x_{s-1}^*) + \sigma_{s-1}(x_{s-1}^* - x_{s-1}) + \sigma_{s-1}(x_{s-1} - x_s^*)
		\\
		=\ & \sigma_{s-1}(\xu_{s-1} - x_{s-1}^*) + \sigma_{s-1}(x_{s-1}^* - x_{s-1}) + \sigma_s(x_{s-1} - x_s^*).
	\end{align*}
	The above recursive relation yields
	$
	\sigma_s(\xu_s - x_s^*) = \sum_{i=1}^{s}[\sigma_{i-1}(x_{i-1}^* - x_{i-1}) + \sigma_i(x_{i-1} - x_i^*)].
	$
	By the above equality and our previous result \eqref{eq:dist_subproblem_unconstrained}, we conclude \eqref{eq:dist_xu_xstarsig_unconstrained}.
\qedsymbol\end{proof}
\vgap

Since the above proposition shows the claim in \eqref{eq:claim_on_distance_unconstrained} holds, we are now ready to estimate the gradient norm $\|\nabla f(x_s)\|$ based on our previous result in \eqref{eq:gradS_estimate_unconstrained}. 
We complete our analysis of Algorithm~\ref{alg:proposed_unconstrained} in the theorem below. 

\vgap
\begin{theorem}
	\label{thm:conv_proposed_unconstrained}
	Assume that the target accuracy $\varepsilon \le LD$, where $D$ is an upper bound on the distance to the set of optimal solutions such that $D\ge \operatorname{dist}(x_0,X^*):=\min_{x^*\in X^*}\|x_0 - x^*\|$. 	
	In Algorithm~\ref{alg:proposed_unconstrained}, suppose that the parameters are set to
	\begin{align}
		\label{eq:cond_proposed_unconstrained}
		S = 1 + \left\lceil\log_{4}({LD}/{\varepsilon})\right\rceil, \ \sigma_s = 4^{s-2}{\varepsilon}/{D},\ \gamma_s = {3}/{4},\ N_s = \left\lceil 8\sqrt{2c_\cA L/\sigma_s}\right\rceil.
	\end{align}
	If the subroutine $\cA$ for solving subproblems $x_s = \cA(f, \sigma_s, x_{s-1}, N_s)$ satisfies the sublinear convergence property \eqref{eq:A_cond_unconstrained} in Assumption~\ref{assum:A_unconstrained}, Algorithm~\ref{alg:proposed_unconstrained} can compute an approximate solution $x_S$ with $\|\nabla f(x_S)\|\le\varepsilon$ after at most 
	$2(1 + 8\sqrt{2c_\cA})\sqrt{{LD}/{\varepsilon}}$ gradient evaluations of $\nabla f$.
\end{theorem}
\begin{proof}
	Noting that function $f_s$ in the subproblem is $\sigma_s$-strongly convex, we have
	$
	f_s(x_s) - f_s(x_s^*) \ge ({\sigma_s}/{2})\|x_s - x_s^*\|^2,
	$
	which, in view of \eqref{eq:A_cond_unconstrained}, implies that
	\begin{align*}
		\|x_s - x_s^*\|\le \dfrac{1}{N_s}\sqrt{\frac{2c_\cA L}{\sigma_s}}\|x_{s-1} - x_{s}^*\|.
	\end{align*}
	Applying $N_s=\lceil8\sqrt{2c_\cA L/\sigma_s}\rceil$ and  noting \eqref{eq:dist_subproblem_unconstrained} in  
	Proposition~\ref{pro:xstar_dist_shrink_unconstrained}, we have
	\begin{align*}
		\|x_s - x_s^*\| \le ({1}/{8})\|x_{s-1} - x_s^*\| \le ({1}/{8})\|x_{s-1} - x_{s-1}^*\|,\ \forall s = 1,\ldots, S.
	\end{align*}
	Using the above result and noting from \eqref{eq:cond_proposed_unconstrained} that $\sigma_S\le L$ and $\sigma_s = 4\sigma_{s-1}$ for all $s>1$, we then conclude from the estimate of $\|\nabla f(x_s)\|$ (with $s=S$) in  \eqref{eq:gradS_estimate_unconstrained} that
	\begin{align*}
		\|\nabla f(x_S)\| \le\ & 8^{-S}L\|x^* - x_0\|  + \sigma_1\|x^* - x_0\| + \sum_{s=2}^{S}5\cdot 4^{s-2}\sigma_1 8^{-s+1}\|x^* - x_0\|
		\\
		\le\ & (L/2)\cdot 4^{-S}\|x_0 - x^*\| + (9/4)\sigma_1\|x_0 - x^*\|.
	\end{align*}
	Here we choose $x^*\in X^*$ to be an optimal solution such that $\|x_0 - x^*\|\le D$.
	Substituting the choices of $S$ and $\sigma_1$ in \eqref{eq:cond_proposed} we have $\|\nabla f(x_S)\|\le \varepsilon$. 
	
	It remains to estimate the total number of gradient evaluations. Noting from \eqref{eq:cond_proposed_unconstrained} that $\sigma_s\le \sigma_S\le L$ for all $s$, we have $N_s\le 1 + 8\sqrt{2c_\cA L/\sigma_s} \le (1 + 8\sqrt{2c_\cA})\sqrt{L/\sigma_s}$. Recalling $\sigma_s=4\sigma_{s-1}$, we can bound the total number of gradient evaluations by
	\begin{align*}
		\sum_{s=1}^{S}N_s \le & (1 + 8\sqrt{2c_\cA})\sqrt{L}\sum_{s=1}^{S}\dfrac{1}{\sqrt{\sigma_s}} \le (1 + 8\sqrt{2c_\cA})\sqrt{\dfrac{L}{\sigma_1}}\sum_{s=1}^{\infty}2^{-(s-1)} 
		\\
		\le & 2(1 + 8\sqrt{2c_\cA})\sqrt{\dfrac{LD}{\varepsilon}}.
	\end{align*}
\qedsymbol\end{proof}
\vgap

A few remarks are in order for the above result. 
First, the regularization parameter $\sigma_s$ is increasing exponentially and the number of gradient evaluations $N_s$ in each iteration of Algorithm~\ref{alg:proposed_unconstrained} is decreasing exponentially. The algorithm will terminate when the regularization parameter $\sigma_s$ reaches $L$. 
Other choices of parameters are also possible as long as we have exponential increment of $\sigma_s$ and decrement of $N_s$. For example, we may also set $\sigma_s = 2^{s-2}\varepsilon/D$ and $\gamma_s = 1/2$ to attain similar gradient evaluation complexity.
Second, we assume $\varepsilon \le LD$ without loss of generality. Indeed, if $\varepsilon>LD$,  the initial value $x_0$ is already an approximate solution since $\|\nabla f(x_0)\|\le L\|x_0 - x^*\|\le LD<\varepsilon$.
Finally, the above result requires the input of two constants, i.e., the Lipschitz constant $L$ and an upper bound $D$ of the distance $\|x_0-x^*\|$. In Section~\ref{sec:pf}, we will describe how to develop parameter-free algorithms that does not require the knowledge of $L$ or $D$.

\section{Minimizing projected gradient over a simple feasible set}
\label{sec:constrained}
In this section, we show that the accumulative regularization technique we introduce in the previous section can be adapted to problems with constrained feasible sets, and (nested) composite objective function. We consider the problem with simple feasible set $X$ and composite function $\phi$:
\begin{align}
	\label{eq:problem} \min_{x\in X}\{ F(x):=f(x) + \phi(x)\}.
\end{align}
Here $X \subset \R^n$ is a closed convex set and $f$ and $\phi$ are closed convex functions. We still assume that $\nabla f$ is $L$-smooth. In addition, we assume that $X$ and $\phi$ are  relatively simple such that for any $\sigma, \eta\ge 0$ and any $x,\xu\in\R^n$, the following optimization problem can be computed efficiently:
\begin{align}
	\label{eq:xpp}
	x^{++}(\eta, \sigma, \xu; x):=\argmin_{u\in X}\langle \nabla f(x), u\rangle + \phi(u) + \frac{\sigma}{2}\|u - \xu\|^2 + \frac{\eta}{2}\|u - x\|^2.
\end{align}
Throughout this paper we assume that an optimal solution $x^*$ to problem \eqref{eq:problem} exists. 
Clearly, $x^*$ is an optimal solution of \eqref{eq:problem} if and only if the \emph{projected gradient} $G_\eta(x^*) = 0$ for some $\eta>0$, where
\begin{align}
	\label{eq:G}
	G_{\eta}(x) :=\ & \eta(x - x^+(\eta;x))\text{ and }
	\\
	\label{eq:xp}
	x^+(\eta;x):=\ &x^{++}(\eta, 0, x;x) = \argmin_{u\in X}\langle\nabla f(x), u\rangle + \phi(u) + \frac{\eta}{2}\|u - x\|^2\ \forall x\in \R^n.
\end{align}
Here $x^+(\eta;x)$ is usually called the \emph{gradient mapping} in the literature. Throughout this section, we may simply use $x^+$ to denote the gradient mapping when the value of $\eta$ is evident from the context. Moreover, when $X=\R^n$ and $\phi\equiv 0$, we can observe that the projected gradient is exactly the gradient: $G_\eta(x) = \nabla f(x)$. Consequently, while small gradient in terms of $\|\nabla f(x)\|$ is no longer an appropriate measure of optimality for the constrained problem \eqref{eq:problem}, it is natural to look for approximate solutions with small projected gradient in terms of $\|G_\eta(x)\|$. Our goal in this section is to design a first-order algorithm that computes an approximate solution $x$ such that $\|G_\eta(x)\|\le \varepsilon$ for some $\eta>0$, given any accuracy threshold $\varepsilon>0$.
The study on unconstrained problem in the previous section is a special case of our study in this section.

It is worth mentioning a few alternative 
accuracy
measures for constrained problems.
For simplicity, suppose for now that $\phi =0$.
If $X$ is bounded, a well-known accuracy measure is given by
$\max_{u \in X} \langle \nabla f(x), x - u\rangle$ ,
which is often called Wolfe gap in some recent literature.
It tells how much a linear approximation of $f$ can decrease over
$X$ starting from $x$. When $X$ is unbounded, one can possibly 
generalize Wolfe gap to a normalized Wolfe gap:
$\max_{u \in X, \|u-x\| \le 1} \langle \nabla f(x), x - u\rangle$.
By examining the optimality conditions of \eqref{eq:xp},
we can see that small projected gradients often implies
small (normalized) Wolfe gap as long as $x$ is bounded (see \cite{lan2020first}). Therefore, we choose to use
projected gradient as an accuracy measure in this section.

Comparing the definition of gradient mapping $x^+$ in \eqref{eq:xp} with that of $x^{++}$ in \eqref{eq:xpp}, we can observe an extra term $(\sigma/2)\|u - \xu\|^2$ in \eqref{eq:xpp}. 
In some sense,
the point $x^{++}(\eta, \sigma, \xu; x)$ in \eqref{eq:xpp} can be viewed as the gradient mapping of a perturbed modification of problem  \eqref{eq:problem}, in which the function $\phi(\cdot)$ is replaced by $\phi(\cdot) + (\sigma/2)\|\cdot - \xu\|^2$. 
Throughout this paper we call $x^{++}(\eta, \sigma, \xu; x)$ the \emph{perturbed gradient mapping} with respect to the \emph{regularization parameter} $\sigma$ and  \emph{prox-center} $\xu$. We may use $x^{++}$ when $\eta,\sigma$ and $\xu$ are evident in the context.

Similar to the previous section, the basic idea for the design of algorithm for the constrained setting is to bound the norm of projected gradient by distances $x_i - x_i^*$, $i=0, \ldots, s$. The overarching picture of our analysis is analogous to the discussion around	\eqref{eq:xs_norm_bound_unconstrained}--\eqref{eq:gradS_estimate_unconstrained} in the previous section, with gradients replaced by projected gradients. Indeed, in Proposition~\ref{pro:grad_dist_relation} we will show a result that is analogous to the gradient-norm-to-distance relationship in \eqref{eq:xs_norm_bound_unconstrained}. The proposition is the consequence of a few properties of (perturbed) gradient mappings in the following lemma.

\vgap

\begin{lem}
	\label{lem:grad_map_properties}
	For any $x, \xu\in\R^n$ and $\eta>0$, $\sigma\ge0$, we have the following properties for the (perturbed) gradient mappings:
	\begin{align}
		\label{eq:xp_lip}
		& \|x^+(\eta+\sigma; x) - x^{++}(\eta, \sigma, \xu; x)\| \le \frac{\sigma}{\eta+\sigma}\|x - \xu\|;
		\\
		\label{eq:xp_monotonicity}
		& \|G_\eta(x)\| \le \|G_{\sigma + \eta}(x)\|.
	\end{align}
	Moreover, if $x^{++}:=x^{++}(\eta, \sigma, \xu; x)$ satisfies 
	\begin{align}
		\label{eq:local_lip}
		f(x^{++}) - f(x) - \langle \nabla f(x), x^{++} - x\rangle \le \frac{M}{2}\|x^{++} - x\|^2
	\end{align}
	for some $M>0$,  we have
	\begin{align}
		\label{eq:3point}
		\begin{aligned}
			&  \left[f(x^{++}) + \phi(x^{++}) + \frac{\sigma}{2}\|x^{++} - \xu\|^2\right] - \left[f(u) + \phi(u) + \frac{\sigma}{2}\|u - \xu\|^2\right]
			\\
			\le\ & \frac{\eta}{2}\|u - x\|^2 - \frac{\sigma + \eta}{2}\|u - x^{++}\|^2 - \frac{\eta-M}{2}\|x^{++} - x\|^2 ,\ \forall u\in X.
		\end{aligned}
	\end{align}
\end{lem}

\begin{proof}
	Let us denote $x^+:=x^+(\eta+\sigma; x)$, $x^{++}:=x^{++}(\eta, \sigma, \xu; x)$ and $x' := x^+(\eta; x)$. By the optimality conditions of \eqref{eq:xp} and \eqref{eq:xpp} we have 
	\begin{align}
		\langle \nabla f(x) + (\eta+\sigma)(x^+ - x), u - x^{+}\rangle + \phi(u) - \phi(x^{+}) \ge 0, \label{eq:oc_simple_1} \\
		\langle \nabla f(x) + \sigma (x^{++}-\bar x) +\eta (x^{++} - x), u - x^{++}\rangle + \phi(u) - \phi(x^{++})\ge 0, \label{eq:oc_simple_2} \\
		\langle \nabla f(x) + \eta (x' - x), u - x'\rangle + \phi(u) - \phi(x') \ge 0 \label{eq:oc_simple_3}
	\end{align}
	for any $u \in X$.
	To prove \eqref{eq:xp_lip}, let us set $u = x^{++}$ in \eqref{eq:oc_simple_1}, $u=\xp$ in \eqref{eq:oc_simple_2}, and add the two relations together. It is easy to see that
	$
	\langle (\eta+\sigma) (x^+ - x^{++})
	+ \sigma (\bar x - x), x^{++} - x^+\rangle  \ge 0 
	$, or equivalently,
	$
	\sigma \langle \bar x - x, x^{++} - x^+\rangle \ge (\eta+\sigma) \|x^+ - x^{++}\|^2,
	$
	which implies \eqref{eq:xp_lip} by the Cauchy-Schwarz inequality.
	Now notice that by definition
	the relation \eqref{eq:xp_monotonicity} is equivalent to $\eta\|x - x'\|\le (\sigma + \eta)\|x - x^{+}\|$.
	Let us set $u = x'$ in \eqref{eq:oc_simple_1}, $u=\xp$ in \eqref{eq:oc_simple_3}, and add the two relations together.
	Then we have 
	$\sigma \langle x^+ - x, x' - x^+\rangle
	- \eta \|x^+ - x'\|^2 \ge 0$,
	which implies 
	$\eta \|x^+ - x'\| \le \sigma \|x - x^+ \|$ again by the Cauchy-Schwarz inequality.
	The previous inequality together with the triangular inequality then immediately implies
	\[
	\eta\|x - x'\| 
	\le \eta (\|x - x^+\|
	+\|x^+ - x'\|) \le (\sigma + \eta) \|x - x^+ \|.
	\]
	In order to show \eqref{eq:3point}, we first notice that
	\eqref{eq:oc_simple_2} is equivalent to 
	
	\begin{align*}
		\begin{aligned}
			\langle \nabla f(x), x^{++} - u\rangle + \phi(x^{++}) - \phi(u)
			&\le ({\sigma}/{2})(\|u - \xu\|^2 - \|x^{++} - \xu\|^2 - \|u - x^{++}\|^2) \\
			& + ({\eta}/{2})(\|u - x\|^2 - \|x^{++} - x\|^2 - \|u - x^{++}\|^2)
		\end{aligned}
	\end{align*}	
	for any $u\in X$.
	Moreover, from \eqref{eq:local_lip} and the convexity of $f$ we have
	\begin{align*}
		f(\xpp) &\le f(x) + \langle \nabla f(x), \xpp - x\rangle + ({M}/{2})\|\xpp - x\|^2\\
		&\le f(x) + \langle \nabla f(x), u - x\rangle 
		+ \langle \nabla f(x), \xpp - u\rangle
		+({M}/{2})\|\xpp - x\|^2\\
		&\le f(u) + \langle \nabla f(x), \xpp - u\rangle
		+({M}/{2})\|\xpp - x\|^2.
	\end{align*} 
	Combining the above two relations we conclude \eqref{eq:3point}.
\qedsymbol\end{proof}

\vgap

With the properties of (perturbed) gradient mapping in the previous lemma, we are ready to present the following proposition that links the gradient mapping with the distance to the optimal solution of the perturbed problem.

\vgap

\begin{proposition}
	\label{pro:grad_dist_relation}
	For any regularization parameter $\sigma>0$ and prox-center $\xu\in \R^n$, if the perturbed gradient mapping $x^{++}:=x^{++}(\eta, \sigma, \xu; x)$ satisfies \eqref{eq:local_lip} and $\eta\ge 2M$, we have
	\begin{align}
		\label{eq:grad_dist_relation}
		\|G_\eta(x)\|\le (3\sigma+2\eta)\|x^*(\sigma;\xu) - x\| + \sigma\|x^*(\sigma;\xu) - \xu\|,
	\end{align}
	where 
	$
	x^*(\sigma;\xu) := \argmin_{u\in X} f(u) + \phi(u) + ({\sigma}/{2})\|u - \xu\|^2
	$
	is the optimal solution to the perturbed problem.
\end{proposition}
\begin{proof}
	Using result \eqref{eq:3point} in Lemma~\ref{lem:grad_map_properties} (with $u = x^*(\sigma;\xu)$) and noting that $\eta\ge 2M$, we have
	\begin{align*}
		0 \le\ &  \left[f(x^{++}) + \phi(x^{++}) + ({\sigma}/{2})\|x^{++} - \xu\|^2\right] 
		\\
		&\ - \left[f(x^*(\sigma;\xu)) + \phi(x^*(\sigma;\xu)) + ({\sigma}/{2})\|x^*(\sigma;\xu) - \xu\|^2\right] 
		\\
		\le\ & ({\eta}/{2})\|x^*(\sigma;\xu) - x\|^2 - (({\sigma + \eta})/{2})\|x^*(\sigma;\xu) - x^{++}\|^2 - ({(\eta-M)}/{2})\|x^{++} - x\|^2
		\\
		\le \ & ({\eta}/{2})\|x^*(\sigma;\xu) - x\|^2 - (({\sigma + \eta})/{2})\|x^*(\sigma;\xu) - x^{++}\|^2 - ({\eta}/{4})\|x^{++} - x\|^2.
	\end{align*}
	Hence $\|\xpp - x\|\le 2\|x^*(\sigma;\xu) - x\|.$
	Also, denoting $x^+:=x^+(\eta;x)$, by the above relation and results \eqref{eq:xp_monotonicity} and \eqref{eq:xp_lip} in  Lemma~\ref{lem:grad_map_properties} we have
	\begin{align*}
		&\quad\ \|G_\eta(x)\| 
		\le
		\|G_{\sigma+\eta}(x)\| = (\sigma+\eta)\|x - x^+(\sigma+\eta;x)\| 
		\\
		& \le (\sigma+\eta)\|x - \xpp\| + (\sigma+\eta)\|\xpp - x^+(\sigma+\eta;x)\|	
		\\
		& \le 2(\sigma+\eta)\|x^*(\sigma;\xu) - x\| + \sigma\|x - \xu\|
		\\
		& \le (3\sigma+2\eta)\|x^*(\sigma;\xu) - x\| + \sigma\|x^*(\sigma;\xu) - \xu\|.
	\end{align*}
	
\qedsymbol\end{proof}

\vgap

The result obtained in \eqref{eq:grad_dist_relation} implies that our proposed Algorithm~\ref{alg:proposed_unconstrained} may be modified to solve problem \eqref{eq:problem}. The modified algorithm is presented in Algorithm~\ref{alg:proposed}.

\begin{algorithm}[h]
	\caption{\label{alg:proposed} An accumulative regularization method for solving problem \eqref{eq:problem}}
	\begin{algorithmic}
		\State In the $s$-th iteration of Algorithm~\ref{alg:proposed_unconstrained}, 
		compute an approximate solution $x_s\in X$ of 
		\begin{align}
			\label{eq:subproblem}
			x_s^*:=\argmin_{x\in X} \{F_s(x):=f(x) + \phi(x) + \frac{\sigma_s}{2}\|x - \xu_s\|^2\},
		\end{align}
		instead of subproblem~\eqref{eq:subproblem_unconstrained}, by running subroutine $\cA(f, \phi, \sigma_s, \xu_s, X, x_{s-1}, N_s)$ with $N_s$ gradient evaluations of $\nabla f$
		and initial point $x_{s-1}$.
	\end{algorithmic}
\end{algorithm}

Two remarks are in place for Algorithm~\ref{alg:proposed}. 
First, $F_s$ in the generalized accumulative proximal mapping subproblem \eqref{eq:subproblem} depends on the regularization parameter $\sigma_s$ and prox-center $\xu_s$, where $\xu_s$ is the convex combination of previous approximate solutions $x_0,\ldots,x_{s-1}$. 
Second, throughout our analysis we will assume that the algorithm $\cA$ has the following performance guarantees, which can be achieved, e.g., by the accelerated gradient method \cite{nesterov2018lectures}.

\begin{assumption}
	\label{assum:A_cond}
	Algorithm $\cA$ used to compute $x_s=\cA(f, \phi, \sigma_s, \xu_s, X, x_{s-1}, N_s)$ has the following performance guarantees after $N_s$ gradient evaluations of $\nabla f$:
	\begin{align}
		\label{eq:A_cond}
		F_s(x_s) - F_s(x_s^*) \le ({c_{\cA}L}/{N_s^2})\|x_{s-1} - x_s^*\|^2.
	\end{align} 
	Here $c_\cA$ is a universal constant that depends on algorithm $\cA$. 
\end{assumption}

We are now ready to study the convergence properties of Algorithm~\ref{alg:proposed}.
Throughout our analysis, we will use $x_s^*$ to denote the optimal solution to the subproblem \eqref{eq:subproblem} for all $s=0,\ldots, S$. In the case when $s=0$, we will set $\sigma_0=0$, so that $x_0^*=x^*$ is an optimal solution to the original problem \eqref{eq:problem}.
We start with the following proposition analogous to Proposition~\ref{pro:xstar_dist_shrink_unconstrained} regarding how 
the distances between $x_s$ and $x^*_s$ are related.
We skip its proof since it could be derived in a straightforward manner based on the proof of Proposition~\ref{pro:xstar_dist_shrink_unconstrained}.

\begin{proposition}
	\label{pro:xstar_dist_shrink}
	In  Algorithm~\ref{alg:proposed} we have
	\begin{align}
		\label{eq:dist_subproblem}
		& \|x_{s-1} - x_s^*\|\le \|x_{s-1} - x_{s-1}^*\|,\ \forall s = 1,\ldots, S\text{ and }
		\\
		\label{eq:dist_xu_xstarsig}
		&\sigma_s\|\xu_s - x_s^*\| 
		\le \sum_{i=1}^{s}(\sigma_{i-1} + \sigma_{i})\|x_{i-1}^* - x_{i-1}\|
		.
	\end{align}
\end{proposition}

With the help of Propositions \ref{pro:grad_dist_relation} and \ref{pro:xstar_dist_shrink}, 
we are now ready to prove the convergence properties of Algorithm~\ref{alg:proposed} in Theorem~\ref{thm:conv_proposed} below.

\vgap

\begin{theorem}
	\label{thm:conv_proposed}
	Assume that the accuracy threshold $\varepsilon \le LD$, where $D$ is an upper bound on the distance to the optimal solution set such that $D\ge \operatorname{dist}(x_0,X^*):=\min_{x^*\in X^*}\|x_0 - x^*\|$. 	
	In Algorithm~\ref{alg:proposed}, suppose that the parameters are set to
	\begin{align}
		\label{eq:cond_proposed}
		S = 2 + \left\lceil\log_{4}({LD}/{\varepsilon})\right\rceil,  \sigma_s = 4^{s-3}{\varepsilon}/{D},\ \gamma_s=  {3}/{4},\ N_s = \left\lceil 8\sqrt{2c_\cA L/\sigma_s}\right\rceil.
	\end{align}
	If the algorithm $\cA$ for solving subproblems satisfies the convergence property \eqref{eq:A_cond} in Assumption~\ref{assum:A_cond}, 
	 we can compute an approximate solution $x_S$ with $\|G_{2L}(x_S)\|\le\varepsilon$ after 
	$2(1 + 8\sqrt{2c_\cA})\sqrt{{LD}/{\varepsilon}}$ gradient evaluations of $\nabla f$.
\end{theorem}
\begin{proof}
	Noting that function $F_s$ in the subproblem is $\sigma_s$-strongly convex, we have
	\begin{align*}
		F_s(x_s) - F_s(x_s^*) \ge \langle F_s'(x_s^*), x_s - x_s^*\rangle + \frac{\sigma_s}{2}\|x_s - x_s^*\|^2
		\ge\ & \frac{\sigma_s}{2}\|x_s - x_s^*\|^2,
	\end{align*}
	where the last inequality above follows from the optimality condition of $x_s^*$. Combining the above relation and the assumption on the sublinear rate of convergence of $\cA$ in \eqref{eq:A_cond}, we have
	$\|x_s - x_s^*\|\le ({1}/{N_s})\sqrt{({2c_\cA L}/{\sigma_s})}\|x_{s-1} - x_{s}^*\|.$
	Applying the selection of $N_s$ in \eqref{eq:cond_proposed} to the above relation and noting \eqref{eq:dist_subproblem} in  Proposition~\ref{pro:xstar_dist_shrink}, we have
	\begin{align}
		\|x_s - x_s^*\| \le \frac{1}{8}\|x_{s-1} - x_s^*\| \le \frac{1}{8}\|x_{s-1} - x_{s-1}^*\|,\ \forall s = 1,\ldots, S.
	\end{align}
	Also, by Proposition~\ref{pro:grad_dist_relation} (with $M = L$ and $\eta = 2L$) and Proposition~\ref{pro:xstar_dist_shrink} we have
	\begin{align}
		\|G_{2L}(x_S)\|\le (3\sigma_S + 4L)\|x^*_S - x_S\| + \sum_{s=1}^{S}(\sigma_{s-1} + \sigma_{s})\|x_{s-1}^* - x_{s-1}\|.
	\end{align}
	The above results, together with the facts that $\sigma_S\le L$ and $\sigma_s = 4\sigma_{s-1}$ for all $s>1$ due to \eqref{eq:cond_proposed}, then imply that
	\begin{align*}
		\|G_{2L}(x_S)\| \le\ & (3L + 4L)8^{-S}\|x^* - x_0\|  + \sigma_1\|x^* - x_0\| + \sum_{s=2}^{S}5\cdot 4^{s-2}\sigma_1 8^{-s+1}\|x^* - x_0\|
		\\
		\le\ & (7/4)L\cdot 4^{-S}\|x_0 - x^*\| + (9/4)\cdot \sigma_1\|x_0 - x^*\|.
	\end{align*}
	Substituting the choices of $S$ and $\sigma_1$ in \eqref{eq:cond_proposed} we have $\|G_{2L}(x_S)\|\le \varepsilon$. 
	It remains to estimate the total number of gradient evaluations. Noting from \eqref{eq:cond_proposed_unconstrained} that $\sigma_s\le \sigma_S\le L$ for all $s$, we have $N_s\le 1 + 8\sqrt{2c_\cA L/\sigma_s} \le (1 + 8\sqrt{2c_\cA})\sqrt{L/\sigma_s}$. Recalling that $\sigma_s=4\sigma_{s-1}$, the total number of gradient evaluations is bounded by
	\begin{align*}
		\sum_{s=1}^{S}N_s \le & (1 + 8\sqrt{2c_\cA})\sqrt{L}\sum_{s=1}^{S}\frac{1}{\sqrt{\sigma_s}} \le (1 + 8\sqrt{2c_\cA})\sqrt{\frac{L}{\sigma_1}}\sum_{s=1}^{\infty}2^{-(s-1)} 
		\\
		\le & 2(1 + 8\sqrt{2c_\cA})\sqrt{\frac{LD}{\varepsilon}}.
	\end{align*}
\qedsymbol\end{proof}
\vgap

A few remarks are in place. 
First, our result implies that in order to compute an approximate solution $x_S$ with a small projected gradient norm $\|G_{2L}(x_S)\|$, the number of gradient evaluations of $\nabla f$ can be bounded by $\cO(1)\sqrt{L D/\varepsilon}$. 
Second, the key ingredient of the analysis in this section is Proposition~\ref{pro:grad_dist_relation}. While its analogy \eqref{eq:xs_norm_bound_unconstrained} in the unconstrained case is straightforward, to extend such a result to projected gradients requires some properties of projected gradients described in Lemma~\ref{lem:grad_map_properties}. 
A few remarks are in order for the above result. 
\cmo{ Third, similar to the results in Theorem \ref{thm:conv_proposed_unconstrained}, here
the regularization parameter $\sigma_s$ starts at $\cO(\varepsilon/D)$ and is increasing exponentially until it reaches $L$. Finally, in our analysis we choose $\gamma_s = 3/4$. Other parameters, e.g., $\gamma_s = 1/2$ can also be chosen to attain similar gradient evaluation complexity.}

In the remaining part of this section, we briefly discuss
how to generalize our algorithm for
a simple nested composite optimization problem of the following form:
\begin{align}
	\label{eq:composite_problem}
	\min_{x\in X} \psi(f(x)) + \phi(x).
\end{align}
Here $X$ is a closed convex set, $\psi:\R\to\R$, $f:\R^n\to\R$, and $\phi:\R^n\to\R$ are closed convex functions, and $\psi$ is monotone non-decreasing (so that the problem is a convex optimization problem). \cmo{ Here we assume $\psi:\R\to\R$ and $f:\R^n\to\R$ only for simplicity of our discussion; all the analysis could be extended to more generalize problems of form $\psi:\R^m \to\R$ and $f:\R^n\to\R^m$.}
The difference between the above problem and problem \eqref{eq:problem} 
exists in the nested structure with $\psi$. 
Without loss of generality we also assume that the Lipschitz constant of $\psi$ is $1$. Moreover, we assume that $\psi$ is relatively simple such that the following \emph{generalized perturbed gradient mapping} can be computed efficiently:
\begin{align}
	\label{eq:xpp_c}
	\begin{aligned}
		& \xpp(\eta,\sigma,\xu;x)
		\\
		:=\ &\argmin_{u\in X}\psi(f(x) + \langle \nabla f(x), u-x\rangle) + \phi(u) + \frac{\sigma}{2}\|u - \xu\|^2 + \frac{\eta}{2}\|u - x\|^2.
	\end{aligned}
\end{align}
We still use the terminology that $\sigma$ and $\xu$ are the regularization parameter and prox-center of the above generalized perturbed gradient mapping. Clearly, $x^*\in X$ is an optimal solution to problem \eqref{eq:composite_problem} if and only if $G_{\eta}(x^*)= 0$ for some $\eta>0$, where we denote the \emph{generalized gradient mapping} and \emph{generalized projected gradient}  as 
\begin{align}
	\label{eq:xp_c}
	x^+(\eta;x):=\ & \argmin_{u\in X}\psi(l_f(x;u)) + \phi(u) + \frac{\eta}{2}\|u - x\|^2\text{ and }
	\\
	\label{eq:G_c}
	G_\eta(x):=\ & \eta(x - x^+(\eta;x))
\end{align}
respectively, where
\begin{align}	
	\label{eq:lf}
	l_f(x;u):=f(x) + \langle \nabla f(x), u-x\rangle
\end{align}
is the linear approximation of $f(x)$ around $x$.
Clearly, when $\psi$ is the identity function, problem \eqref{eq:composite_problem} reduces to problem \eqref{eq:problem}, and the definitions in \eqref{eq:xp_c} and \eqref{eq:G_c} reduces to \eqref{eq:xp} and \eqref{eq:G} respectively. 

Our proposed algorithm to solve the nested composite optimization problem \eqref{eq:composite_problem} is a straightforward extension of Algorithm~\ref{alg:proposed}, obtained by replacing subproblem \eqref{eq:subproblem} with 
\begin{align}
	\label{eq:subproblem_c}
	x_s^*:=\argmin_{x\in X} \{ F_s(x):=\psi(f(x)) + \phi(x) + \frac{\sigma_s}{2}\|x - \xu_s\|^2\}.
\end{align} 
Moreover, we need to assume that 
Assumption~\ref{assum:A_cond} about the performance of
Algorithm $\cA$ applies to the subproblem~\eqref{eq:subproblem_c}.
While the subproblem \eqref{eq:subproblem_c} has a nested structure, due to the simplicity of $\psi$ the above convergence property can be achieved through a nested optimization extension of the accelerated gradient method (with $c_\cA = 2)$ that only requires computations of generalized gradient mappings (see, e.g.,  \cite{zhang2020optimal,nesterov2018lectures,lan2015bundle}).
We state in the following theorem on the convergence properties of this more general algorithm, and
its proof is a straightforward modification of that of Theorem~\ref{thm:conv_proposed} and hence is skipped.

\begin{theorem}
	\label{thm:conv_proposed_c}
	Assume that the accuracy threshold $\varepsilon \le LD$, where $D$ is an upper bound on the distance to the optimal solution set such that $D\ge \operatorname{dist}(x_0,X^*):=\min_{x^*\in X^*}\|x_0 - x^*\|$. 	
	In Algorithm~\ref{alg:proposed} (with subproblem replaced by \eqref{eq:subproblem_c}), suppose that the parameters are set to
	\begin{align}
		S = 2 + \left\lceil\log_{4}({LD}/{\varepsilon})\right\rceil,  \sigma_s = 4^{s-3}{\varepsilon}/{D},\ \gamma_s = {3}/{4},\ N_s = \left\lceil{8\sqrt{2c_\cA L\sigma_s}}\right\rceil.
	\end{align}
	If the algorithm $\cA$ for solving \cmo{subproblems} satisfies the  convergence property \eqref{eq:A_cond} in Assumption~\ref{assum:A_cond}, 
	 we can compute an approximate solution $x_S$ with $\|G_{2L}(x_S)\|\le\varepsilon$ after 
	$2(1 + 8\sqrt{2c_\cA})\sqrt{{LD}/{\varepsilon}}$ gradient evaluations of $\nabla f$.
\end{theorem}
\vgap

\section{Parameter-free algorithm for convex problems}
\label{sec:pf}

For all the convergence results presented in the previous sections, we need  to know some important problem parameters including the Lipschitz constant $L$ and an upper bound $D$ on the distance $\operatorname{dist}(x_0, X^*)$. In this section, we study parameter-free implementation strategies that do not require such information. For simplicity, we will focus on the unconstrained case, but the parameter-free implementation strategies can be derived similarly for the cases with simple feasible set and (nested) composite objective function.

We first study how to remove the requirement on the information of $L$.
Observe that for Algorithm~\ref{alg:proposed_unconstrained} and its convergence analysis in Section~\ref{sec:proposed_unconstrained}, $L$ appears 
for the first time in the gradient-to-distance relationship in \eqref{eq:xs_norm_bound_unconstrained}. In the lemma below, we show that the constant $L$ in \eqref{eq:xs_norm_bound_unconstrained} can be replaced by its local estimate. 

\vgap
\begin{lemma}
	\label{thm:Mlemma}
	Suppose that $g:\R^n\to\R$ is a convex smooth function.
	For any $x\in\R^n$, $M>0$, $\eta\ge 2M$, and $\xpp:=x - (1/\eta)\nabla g(x)$,
	if $\xpp$ and $x$ satisfy
	\begin{align}
		\label{eq:M}
		g(\xpp) - g(x) - \langle \nabla g(x), \xpp - x\rangle \le \frac{M}{2}\|\xpp - x\|^2,
	\end{align}
	then we have $\|\nabla g(x)\|\le \sqrt{2}\eta \|x - x^*\|$, where $x^*$ is any minimizer of function $g$.
\end{lemma}
\begin{proof}
	Note that our definition of $\xpp$ is exactly $\xpp(\eta,0,\xu;x)$ described in \eqref{eq:xpp} (with $X\in\R^n$, $\sigma=0$, and $\phi\equiv 0$). Applying result \eqref{eq:3point} in Lemma~\ref{lem:grad_map_properties} and noting that $\eta\ge 2M$ we have
	\begin{align*}
		g(\xpp) - g(u) \le \frac{\eta}{2}\|u - x\|^2 - \frac{\eta}{2}\|u - \xpp\|^2 - \frac{\eta}{4}\|\xpp - x\|^2, \ \forall u\in \R^n.
	\end{align*}
	Letting $u=x^*$ in the above relation, using the optimality of $x^*$ and the fact that $\xpp - x = -(1/\eta)\nabla g(x)$ we conclude that 
	$	\|\nabla g(x)\|^2 \le 2\eta^2\|x^* -x\|^2$.
\qedsymbol\end{proof}

\vgap

\begin{algorithm*}[h] 
	\caption{$M$=Backtracking($g$, $\sigma$, $x$, $M_0$) \label{alg:_searchL}}
	\begin{algorithmic}
		\For{$j=0,1,\ldots,$}
		\State Set
		$			\xpp = x - ({1}/{(2(M_j + \sigma))})\nabla g(x).
		$
		\State If 
		\begin{align}
			\label{eq:linesearch_estL_success_criterion}
			g(\xpp) - g(x) - \langle \nabla g(x), \xpp - x\rangle \le ({(M_j + \sigma)}/{2})\|\xpp - x\|^2, 
		\end{align}
		then {\bf terminate} with $M=M_j$. Otherwise, set $M_{j+1} = 2M_{j}$. 
		\EndFor
	\end{algorithmic}
\end{algorithm*}

The above lemma allows us to design a backtracking function described in Algorithm~\ref{alg:_searchL}. Applying Lemma~\ref{thm:Mlemma} with $x = x_s$, $x^* = x_s^*$, $g = f_s$, $M=M_s+\sigma_s$, and $\eta = 2(M_s+\sigma_s)$, we conclude that whenever \eqref{eq:M} holds, 
we have $\|\nabla f_s(x_s)\|\le 2\sqrt{2}(M_s+\sigma_s)\|x_s - x_s^*\|$. As a consequence,
\begin{align}
	\label{eq:xs_norm_bound_unconstrained_est}
	\begin{aligned}
		& \|\nabla f(x_s)\| = \|\nabla f_s(x_s) - \sigma_s(x_s - \xu_s)\| 
		=  \|\nabla f_s(x_s) - \sigma_s(x_s - x_s^*) - \sigma_s(\xu_s-x_s^*)\| 
		\\
		\le &  (2\sqrt{2}M_s+(2\sqrt{2}+1)\sigma_s)\|x_s - x_s^*\| + \sigma_s\|x_s^* - \xu_s\|,
	\end{aligned}
\end{align}
which is analogous to the gradient-to-distance relation \eqref{eq:xs_norm_bound_unconstrained} but with $L$ replaced by $M_s$.
Apparently, we can apply one single gradient descent step with backtracking, i.e., Algorithm~\ref{alg:_searchL} with input $g=f_s$, $\sigma =\sigma_s$, $x= x_s$ and initial guess $M_0$, to estimate the local Lipschitz constant $M_s$.

We are now ready to modify Algorithm~\ref{alg:proposed_unconstrained} to remove its requirement on the knowledge of $L$ entirely.
This proposed accumulative regularization (AR) method
(see Algorithm~\ref{alg:proposed_unconstrained_searchL}) can be called as a function $(\hat x,M)=AR(f,x_0,\sigma_1, M_0)$. Here the input arguments $x_0$, $\sigma_1$ and $M_0$ denote the initial point, initial regularization parameter, and initial guess on the Lipschitz constant $L$, respectively. Once terminated, this method returns an approximate solution with small gradient norm and an updated estimate of the Lipschitz constant, denoted by $\hat x$ and $M$, respectively.

\begin{algorithm}[!h]
	\caption{\label{alg:proposed_unconstrained_searchL} Accumulative regularization (AR) without the knowledge of $L$}
	\begin{algorithmic}
		\Function{$(\hat x, M)=\ $AR}{$f, x_0, \sigma_1, M_0$}
		\State Set the initial prox-center $\xu_0 = x_0$, and denote $\sigma_0=0$.
		\For {$s=1,2,\ldots$}
		\State If $s>1$, then set $\sigma_s=4\sigma_{s-1}$.
		\State Set $\xu_s =  (1-\gamma_s) \xu_{s-1} + \gamma_s x_{s-1}$ with $\gamma_s=1 - \sigma_{s-1}/\sigma_s$.
		
		\State Compute an approximate solution $x_s\in \R^n$ of
		\begin{align}
			\label{eq:subproblem_unconstrained_searchL}
			x_s^*:=\argmin_{x\in \R^n} \{ f_s(x):=f(x) + \dfrac{\sigma_s}{2}\|x - \xu_s\|^2\}
		\end{align}
		by running subroutine $x_s = \cA(f, \sigma_s, \xu_s, x_{s-1})$ with initial point $x_{s-1}$. 
		
		\State Set $M_s=$ Backtracking ($f_s$, $\sigma_s$, $x_s$, $M_{s-1}/2$).
		\State If $\sigma_s\ge M_s$, then {\bf terminate} with $\hat x=x_s$ and $M=M_s$ . 
		\EndFor
		\EndFunction
	\end{algorithmic}
\end{algorithm}

It is worth noting several differences between our proposed Algorithm~\ref{alg:proposed_unconstrained_searchL} and the previous Algorithm~\ref{alg:proposed_unconstrained}. 
First, after computing an approximate solution $x_s$ to subproblem~\eqref{eq:subproblem_unconstrained_searchL}, we 
call the function {\sl Backtracking} in~Algorithm~\ref{alg:_searchL} to estimate a bound $M_s$ for the Lipschitz constant $L$, which will ensure the relation in \eqref{eq:xs_norm_bound_unconstrained_est}. The estimates $\{M_s\}$ satisfy $M_s\le 2L$.

Second, in Algorithm~\ref{alg:proposed_unconstrained_searchL} we no longer require a maximum number of subproblems $S$ to be given as in Algorithm~\ref{alg:proposed_unconstrained}. Rather, we will keep increasing $s$ until termination criterion $\sigma_s\ge M_s$ holds. The algorithm is guaranteed to terminate since $\sigma_s$ is increasing exponentially and $M_s\le 2L$.

Finally, unlike Algorithm~\ref{alg:proposed_unconstrained}, when calling subroutine $\cA$ to compute $x_s$ we no longer provide the number of gradient evaluations $N_s$. Instead, 
we make the following assumption regarding the convergence properties of algorithm $\cA$.

\vgap

\begin{assumption}
	\label{assum:A_unconstrained_searchL}
	The subroutine $x_s =\cA(f, \sigma_s, \xu_s, x_{s-1})$ has the following performance guarantees after the $k$-th gradient evaluation of $\nabla f$:
	\begin{align}
		\label{eq:A_cond_unconstrained_searchL}
		f_s(x_s^k) - f_s(x_s^*) \le ({L_s^k}/{k^2})\|x_{s-1} - x_s^*\|^2.
	\end{align} 
	Here $x_s^k$ is the computed approximate solution and $L_s^k$ is an estimate of Lipschitz constant $L$ such that $L_s^k\le c_\cA L$, where $c_\cA$ is a universal constant that depends on algorithm $\cA$. 
	Subroutine $\cA$ will terminate whenever $k\ge 8\sqrt{2L_s^k/\sigma_s}$ and output the approximate solution $x_s = x_s^k$.
\end{assumption}
The above sublinear convergence assumption is satisfied by many algorithms. For example, if we use a line search based uniformly optimal accelerated gradient method in \cite{nesterov2018lectures} (see, e.g.,  \cite{nesterov2015universal}),  the assumption is satisfied with $c_\cA = 4$.

\vspace{0.1in}

We state below in Theorem~\ref{thm:conv_proposed_unconstrained_searchL}
the main convergence properties for Algorithm~\ref{alg:proposed_unconstrained_searchL}.
The proof of this result is analogous to that of Algorithm~\ref{alg:proposed_unconstrained}, but
we need to incorporate the adaptive estimation of $L$ into the overarching picture in \eqref{eq:xs_norm_bound_unconstrained}--\eqref{eq:gradS_estimate_unconstrained}. 

\begin{theorem}
	\label{thm:conv_proposed_unconstrained_searchL} 
		For any input argument $M_0, \sigma_1>0$, the AR method outputs an $\hat x$ with $\|\nabla f(\hat x)\|\le 5\sigma_1\operatorname{dist}(x_0, X^*)$ within $4 + 3\sqrt{\max\{M_0/2, 2L\}/\sigma_1}+16\sqrt{2c_\cA }\sqrt{L/\sigma_1}$ gradient evaluations of $\nabla f$.
		Specially, if $M_0\le 4L$, the total number of gradient evaluations is bounded by $4 + C_1\sqrt{L/\sigma_1}$, where $C_1:= \sqrt{2}(3 + 16\sqrt{2c_\cA})$ is a universal constant. 
		If in addition $\sigma_1 = \varepsilon/(5D)$ for some $D>0$, then an approximate solution $\hat x$ satisfying $\|\nabla f(\hat x)\|\le \varepsilon\operatorname{dist}(x_0, X^*)/D$  can be computed within at most  $4 + \sqrt{5}C_1\sqrt{LD/\varepsilon}$ gradient evaluations of $\nabla f$. 
	\end{theorem}
	\begin{proof}
		First note that by \eqref{eq:dist_xu_xstarsig_unconstrained} in Proposition~\ref{pro:xstar_dist_shrink_unconstrained} and the relation in \eqref{eq:xs_norm_bound_unconstrained_est}, we have 
		\begin{align}
			\label{eq:gradS_estimate_unconstrained_searchL}
			\begin{aligned}
				\|\nabla f(x_s)\|
				&	\le (2\sqrt{2}M_s + (1+2\sqrt{2})\sigma_s)\|x_s -x_s^*\| \\
				&\quad + \sigma_1\|x^* - x_0\| + \sum_{i=2}^s(\sigma_{i-1} + \sigma_i)\|x_{i-1}^* - x_{i-1}\|,
			\end{aligned}
		\end{align}
		which is analogous to \eqref{eq:gradS_estimate_unconstrained}.
		Consider the call to the subroutine $\cA$ for minimizing $f_s(x)$ in \eqref{eq:subproblem_unconstrained_searchL}. Noting that $f_s$ is $\sigma_s$-strongly convex, we have $\|x - x_s^*\|^2 \le ({2}/{\sigma_s})(f_s(x) - f_s(x_s^*)),\ \forall x\in\R^n$. In view of this observation and Assumption~\ref{assum:A_unconstrained_searchL}, after the $k$-th gradient evaluation of $\nabla f$, subroutine $\cA$ will compute an approximation solution $x_s^k$ such that
		\begin{align}
			\label{eq:A_dist_shrink}
			\|x_s^k - x_s^*\|\le\dfrac{1}{k}\sqrt{\dfrac{2L_s^k}{\sigma_s}}\|x_{s-1} - x_{s}^*\|\le  \dfrac{1}{k}\sqrt{\dfrac{2c_\cA L}{\sigma_s}}\|x_{s-1} - x_{s}^*\|.
		\end{align}
		Also by	Assumption~\ref{assum:A_unconstrained_searchL}, the number of gradient evaluations that $\cA$ performs is bounded by 
		\begin{align}
			\label{eq:Ns}
			N_s:=\lceil8\sqrt{2L_s/\sigma_s}\rceil\le \lceil8\sqrt{2c_\cA L/\sigma_s}\rceil.
		\end{align} 
		By the above two relations and \eqref{eq:dist_subproblem_unconstrained} in  
		Proposition~\ref{pro:xstar_dist_shrink_unconstrained}, the output solution $x_s$ computed at the termination of subroutine $\cA$ applied to minimize $f_s$ satisfies
		\begin{align*}
			\|x_s - x_s^*\| \le \frac{1}{8}\|x_{s-1} - x_s^*\| \le \frac{1}{8}\|x_{s-1} - x_{s-1}^*\|,\ \forall s = 1,\ldots, S.
		\end{align*}
		Using the above result and the fact that $\sigma_s = 4\sigma_{s-1}$ for all $s>1$ in our previous estimate of $\|\nabla f(x_s)\|$ in  \eqref{eq:gradS_estimate_unconstrained_searchL}, we have
		\begin{align}
			\label{eq:tmp_grad_est}
			\begin{aligned}
				\|\nabla f(x_s)\| \le\ & (2\sqrt{2}M_s + (1+2\sqrt{2})\sigma_s)8^{-s}\|x^* - x_0\|  
				\\
				&\ + \sigma_1\|x^* - x_0\| + \sum_{i=2}^{s}5\cdot 4^{i-2}\sigma_1 8^{-i+1}\|x^* - x_0\|
				\\
				\le\ & (\sqrt{2}M_s + (1/2+\sqrt{2})\sigma_s) 4^{-s}\|x_0 - x^*\| + (9/4)\sigma_1\|x_0 - x^*\|.
			\end{aligned}
		\end{align}
		Here $x^*$ is any solution in the optimal solution set $X^*$. 
		
		We will now estimate $\sigma_s$ and $4^{-s}$ in the above relation when the termination criterion $\sigma_s\ge M_s$ is satisfied. 
		If the criterion is satisfied when $s=1$, by \eqref{eq:tmp_grad_est} the output $\hat x=x_1$ satisfies
		\begin{align*}
			\|\nabla f(\hat x)\| \le (1/2+2\sqrt{2})\sigma_1 4^{-1}\|x_0 - x^*\| + (9/4)\sigma_1\|x_0 - x^*\|\le 5\sigma_1\|x_0 - x^*\|.
		\end{align*}
		Otherwise, let $S>1$ be the smallest $s$ when the termination criterion $\sigma_s\ge M_s$ is satisfied. Observe in the backtracking function that $M_s\le\max\{M_{s-1}/2, 2L\}$. Thus we have $M_s\le \max\{M_0/2, 2L\}$ for $s=1,\ldots, S$. Use this observation and recalling that $\sigma_s = 4\sigma_{s-1}$, we have 
		\begin{align}
			\label{eq:sigmaS_bound}
			M_S\le 4^{S-1}\sigma_1 = \sigma_{S} = 4\sigma_{S-1}<4M_{S-1}\le \max\{2M_0, 8L\}.
		\end{align}
		Thus by \eqref{eq:tmp_grad_est}, during termination the output $\hat x = x_S$ satisfies
		\begin{align*}
			\|\nabla f(\hat x)\| \le & (1/2+2\sqrt{2})\sigma_S 4^{-S}\|x_0 - x^*\| + (9/4)\sigma_1\|x_0 - x^*\|
			\le 5\sigma_1\|x_0 - x^*\|.
		\end{align*}
		Since $x^*$ here is an arbitrary optimal solution, we have $\|\nabla f(\hat x)\|\le 5\sigma_1\operatorname{dist}(x_0, X^*)$.	
		
		Next we will bound the total number of gradient estimations. 
		If the termination criterion is satisfied when $s=1$,  bounding the numbers of gradient evaluations required for backtracking line-search and subroutine $\cA$ respectively, we have the bound $1 + N_1$. Recalling the value of $N_s$ and its bound in \eqref{eq:Ns}, this number is bounded by $ 1 + 8\sqrt{2c_A L/\sigma_1}$. If the termination criterion is satisfied at some $s=S>1$,  the total number of gradient evaluations is $S+\sum_{s=1}^{S}N_s$. Here by \eqref{eq:sigmaS_bound} we have $S\le 2 + \log_4(\max\{M_0/2, 2L\}/\sigma_1)$. Also, recalling the value of $N_s$ and its bound in \eqref{eq:Ns} and noting that $\sigma_s = 4\sigma_{s-1}$ for all $s$ we have
		\begin{align*}
			\sum_{s=1}^{S}N_s \le \sum_{s=1}^{S}1 + 8\sqrt{2c_\cA L/\sigma_s} \le & S + 8\sqrt{2c_\cA L/\sigma_1}\sum_{s=1}^{\infty}2^{-(s-1)}  \le S + 16\sqrt{2c_\cA L/\sigma_1}.
		\end{align*}      
		Summarizing the above we can bound the total number of gradient evaluations by
		\begin{align}
			\label{eq:total_gradient}
			\begin{aligned}
				& S+\sum_{s=1}^{S}N_s 
				\le 2S + 16\sqrt{2c_\cA L/\sigma_1} \le 2(2 + \log_4(\max\{M_0/2, 2L\}/\sigma_1)) + 16\sqrt{2c_\cA L/\sigma_1}
				\\ 
				\le\ & 4 + 3\sqrt{\max\{M_0/2, 2L\}/\sigma_1}+16\sqrt{2c_\cA }\sqrt{L/\sigma_1}. 
			\end{aligned}
		\end{align}
		Here in the last inequality we use the fact that $2\log_4 t \cmo{<} 3\sqrt{t}$ for all $t>0$.
		Note that the above bound is also an upper bound of number of gradient evaluations if the termination criterion is satisfied when $s=1$. 
		The remainder parts of result then immediately follows from \eqref{eq:total_gradient}.
	\qedsymbol\end{proof}
	\vgap
	
	A few remarks are in place. 
	First, in the last result of the above theorem we require $M_0\le 4L$. This condition can be satisfied if we select any $z_0\not=x_0$ and set $M_0:=\|\nabla f(x_0) - \nabla f(z_0)\|/\|x_0-z_0\|$.
	
	Second, in the above result we only consider the total number of gradient evaluations of $\nabla f$. We may also bound the total number of function evaluations of $f$ based on the above result. To estimate the bound, let us assume that the total number of gradient evaluations of $\nabla f$ and function evaluations of $f$ are the same within each call to the subrountine $\cA$. Under such assumption, we can observe that the only difference between the total number of function evaluations of $f$ and gradient evaluations of $\nabla f$ is the extra function evaluations in the backtracking function whenever the criterion \eqref{eq:linesearch_estL_success_criterion} is not satisfied. For the $s$-th call to the backtracking function, such number of extra function evaluations is bounded by $2+\log_2(M_s/M_{s-1})$. Thus the total number of extra function evaluations is bounded by 
	$
		2S + \log_2(2L/M_0),
	$
	where $S$ is the number $s$ when the AR method terminates. Recalling the bound $S\le 2 + \log_4(\max\{M_0/2, 2L\}/\sigma_1)$ in the proof of the above theorem, if the input arguments of the AR method satisfy $M_0\le 4L$ and $\sigma_1=\varepsilon/(5D)$ for some $D>0$,  the total number of function evaluations of $f$ required by the AR method is bounded by $8 + \sqrt{5}C_1\sqrt{LD/\varepsilon} + 2\log_4(10LD/\varepsilon) + \log_2(2L/M_0)$. This is the same order to the total number of gradient evaluations of $\nabla f$, with an addition of a logarithmic term.
	
	Third, by Theorem~\ref{thm:conv_proposed_unconstrained_searchL}, if we have information on the upper bound of the distance to any optimal solution, i.e., $D\ge \operatorname{dist}(x_0, X^*)$, then by setting $\sigma_1=\varepsilon/(5D)$ the output $\hat x$ satisfies $\|\nabla f(\hat x)\|\le \varepsilon$. The number of gradient evaluations for computing such an approximate solution is given by $\cO(1)\sqrt{LD/\varepsilon}$.
	Note that if $D \gg \operatorname{dist}(x_0, X^*)$,  our complexity result could be worse than the desired $\cO(1)\sqrt{L\operatorname{dist}(x_0, X^*)/\varepsilon}$. 
	However, the case when $D > \operatorname{dist}(x_0, X^*)$ can be avoided entirely by choosing $D=\|\nabla f(x_0)\|/(2\sqrt{2}M)$, where $M=$Backtracking$(f, 0, x_0, M_0)$ (see Algorithm~\ref{alg:_searchL}). Applying Lemma~\ref{thm:Mlemma} with $x=x_0$, $g= f$, and $\eta = 2M$, we have $\|\nabla f(x_0)\|\le 2\sqrt{2}M\|x_0 - x^*\|$ and hence $D\le \operatorname{dist}(x_0, X^*)$.
	
	Fourth, if $D$ is chosen optimistically such that $D\le\operatorname{dist}(x_0, X^*)$, then our output $\hat x$ will not satisfy the desired accuracy threshold $\|\nabla f(\hat x)\|\le\varepsilon$. 
	However, in such case the gradient evaluation complexity $\cO(1)\sqrt{LD/\varepsilon}$ is also small. Therefore, if $\|\nabla f(\hat x)\|\le \varepsilon$ is not satisfied when the algorithm terminates, we may simply discard all results and restart the computation with a larger guess $D$. Such ``guess-and-check'' strategy will maintain our desired gradient complexity while guaranteeing that our output $\hat x$ satisfies $\|\nabla f(\hat x)\|\le \varepsilon$. We describe such a scheme in Algorithm~\ref{alg:proposed_unconstrained_guess_and_check}. 
	
	\begin{algorithm}[h]
		\caption{\label{alg:proposed_unconstrained_guess_and_check} A ``guess-and-check'' implementation of the AR method}
		\begin{algorithmic}
			\Require Target accuracy $\varepsilon>0$ and initial point $x_0\in\R^n$.
			\State Select any $z_0\not= x_0$ \cmo{ such that $\nabla f(x_0) \not= \nabla f(z_0)$}. Set $\tilde M_0 = \|\nabla f(x_0) - \nabla f(z_0)\|/\|x_0-z_0\|$, $M_0 = $Backtracking$(f,0,x_0, \tilde M_0)$ (see Algorithm~\ref{alg:_searchL}), and $D_0 = \|\nabla f(x_0)\|/(2\sqrt{2}M_0)$.
			\For {$t=1,2,\ldots$}
			\State Set $D_t = 4D_{t-1}$.
			\State Set $(\hat x, M_t) = \ $AR$(f, x_0, \varepsilon/(5D_t), M_{t-1})$ (see  Algorithm~\ref{alg:proposed_unconstrained_searchL}).
			\State If $\|\nabla f(\hat x)\|\le \varepsilon$, then {\bf terminate} with $\hat x$. 
			\EndFor
		\end{algorithmic}
	\end{algorithm}
	
	The performance of the above ``guess-and-check'' AR implementation in Algorithm~\ref{alg:proposed_unconstrained_guess_and_check} is stated in the following theorem.
	
	\vgap
	\begin{theorem}
		Suppose that $\|\nabla f(x_0)\|> \varepsilon$ in Algorithm~\ref{alg:proposed_unconstrained_guess_and_check}. The algorithm computes a solution $\hat x$ such that $\|\nabla f(\hat x)\|\le\varepsilon$ with at most $4\lceil\log_4(4\sqrt{2}L\operatorname{dist}(x_0, X^*)/\varepsilon)\rceil+ 4\sqrt{5}C_1\sqrt{L\operatorname{dist}(x_0,X^*)/\varepsilon}$ gradient evaluations of $\nabla f$. 
		Here $C_1$ is a universal constant defined in Theorem~\ref{thm:conv_proposed_unconstrained_searchL}.
	\end{theorem}
	\begin{proof}
		Applying Lemma~\ref{thm:Mlemma} with $x=x_0$, $g= f$, and $\eta = 2M_0$ we have $\|\nabla f(x_0)\|\le 2\sqrt{2}M_0\|x_0 - x^*\|$ and hence $D_0\le \operatorname{dist}(x_0, X^*)$.
		From Theorem~\ref{thm:conv_proposed_unconstrained_searchL}, in the $t$-th call of the AR function we have a solution $\hat x$ such that $\|\nabla f(\hat x)\|\le \varepsilon\operatorname{dist}(x_0, X^*)/D_t$.
		Also note that whenever $t\ge T:=\lceil\log_4(\operatorname{dist}(x_0, X^*)/D_0)\rceil$,  $D_t = 4^tD_0 \ge \operatorname{dist}(x_0, X^*)$.
		By Theorem~\ref{thm:conv_proposed_unconstrained_searchL}, Algorithm~\ref{alg:proposed_unconstrained_guess_and_check} will terminate after at most $T$ calls to the AR function. The total number of gradient evaluations of $\nabla f$ is bounded by $4T+\sum_{t=1}^{T}\sqrt{5}C_1\sqrt{10LD_t/\varepsilon}$, where $T\le \lceil\log_4(4\sqrt{2}L\operatorname{dist}(x_0, X^*)/\varepsilon)\rceil$ and
		\begin{align*}
			& \sum_{t=1}^{T}\sqrt{LD_t/\varepsilon} = \sqrt{LD_0/\varepsilon}\sum_{t=1}^{T}2^t \le 2\sqrt{LD_0 4^T/\varepsilon}\le 4\sqrt{L\operatorname{dist}(x_0,X^*)/\varepsilon}.
		\end{align*}
	\qedsymbol\end{proof}
	\vgap
	
	By the above result, Algorithm~\ref{alg:proposed_unconstrained_guess_and_check} is a parameter-free algorithm that has gradient evaluation complexity of order $\cO(\sqrt{L\operatorname{dist}(x_0, X^*)/\varepsilon})$. To the best of our knowledge, this is the first parameter-free algorithm with optimal complexity for gradient minimization in the literature. Similar to the second remark after Theorem~\ref{thm:conv_proposed_unconstrained_searchL}, we can prove that the total number of function evaluations of $f$ is of the same order, with an addition of a logarithmic term $\log_2(L/M_0)$.
	
	\section{Optimal and parameter-free algorithm for strongly convex problems}
	\label{sec:sc}
	Our goal in this section is to develop an optimal and parameter-free algorithm for $\min_{x\in\R^n}f(x)$, in which $f$ is both $L$-smooth and $\mu$-strongly convex. 
	
	Suppose for now that $\mu$ and $L$ are given. 
	Since $\|\nabla f(\hat x)\|\le L\|\hat x - x^*\|$, 
	we can guarantee $\|\nabla f(\hat x)\|\le \varepsilon$ by computing an approximate solution $\hat x$ with $\|\hat x - x^*\|\le \varepsilon/L$. By applying, for example, the accelerated gradient method \cite{nesterov2018lectures}, we can compute such a solution within at most $\cO(1)\sqrt{L/\mu}\log(L\|\nabla f(x_0)\|/(\mu \varepsilon))$ gradient evaluations.
	It turns out that with the help of improved complexity results on gradient minimization for convex unconstrained problems, we can achieve the optimal 
	$\cO(1)\sqrt{L/\mu}\log(\|\nabla f(x_0)\|/\varepsilon)$ 
	complexity by using a straightforward restarting strategy. Specifically, consider any method (e.g., our proposed AR method or the ones in \cite{kim2021optimizing,diakonikolas2022potential,lee2021geometric}) that has an $\cO(1)\sqrt{L\|x_0 - x^*\|/\varepsilon}$ complexity for gradient minimization. Setting $\varepsilon=\|\nabla f(x_0)\|/2$  and using the fact that $\|x_0 - x^*\|\le \|\nabla f(x_0)\|/\mu$ due to the strong convexity of $f$, we conclude that it takes at most $\cO(1)\sqrt{L/\mu}$ gradient evaluations to obtain a solution $\hat x$ s.t. $\|\nabla f(\hat x)\|\le \|\nabla f(x_0)\|/2$. 
	Therefore, starting from $x_0$, by simply restarting the algorithm whenever the gradient norm is reduced by half, we can obtain an approximate solution $\hat x$ with $\|\nabla f(\hat x)\|\le\varepsilon$ in at most $\cO(1)\log(\|\nabla f(x_0)\|/\varepsilon)$ restarts. The total gradient evaluation is bounded by $\cO(1)\sqrt{L/\mu}\log(\|\nabla f(x_0)\|/\varepsilon)$. 
	
	When the constants $\mu$ and $L$ are not available, in order to employ the above restarting strategy we will need a method that does not need the input of $L$. Unfortunately, the current results \cite{kim2021optimizing,diakonikolas2022potential,lee2021geometric} in the literature does not satisfy such requirement. However, the proposed AR method in the previous section can help us in this regard. First recall in Theorem~\ref{thm:conv_proposed_unconstrained_searchL} that for any input argument $\sigma_1>0$, the AR method outputs a point $\hat x$ with $\|\nabla f(\hat x)\|\le 5\sigma_1\|x_0 - x^*\|$. Setting $\sigma_1 = \mu/10$ and using the strong convexity of $f$, we have $\|\nabla f(\hat x)\|\le \mu\|x_0 - x^*\|/2\le\|\nabla f(x_0)\|/2$. The gradient norm is now reduced by half and we may restart the AR method. When the value of $\mu$ is not available, we may set $\sigma_1 = \tilde \mu/10$ with a guess $\tilde \mu$ and then use a guess-and-check implementation similar to Algorithm~\ref{alg:proposed_unconstrained_guess_and_check} to search the correct $\tilde\mu$. Combining the guess-and-check implementation with the aforementioned restarting scheme, we are able to obtain a parameter-free algorithm with optimal complexity. Our proposed strongly convex accumulative regularization method (SCAR) is described in Algorithm~\ref{alg:proposed_SCAR_noflag}.
	
	\begin{algorithm}[h]
		\caption{\label{alg:proposed_SCAR_noflag} A parameter-free strongly convex accumulative regularization (SCAR) algorithm}
		\begin{algorithmic}
			\Function{($\hat x$, $\hat M$) = SCAR}{$f$, $\varepsilon$, $y_0$, $\mu_0$, $M_0$}
			\For {$t=1,2,\ldots$}
			\State Set
			$
			(y_{t},M_{t}) = \text{AR}(f, y_{t-1}, \mu_{t-1}/10, M_{t-1})
			$ (see Algorithm~\ref{alg:proposed_unconstrained_searchL}).
			\State If {$\|\nabla f(y_t)\|>\|\nabla f(y_{t-1})\|/2$}, then set $\mu_t = \mu_{t-1}/4$ and $y_t = y_{t-1}$. 		
			\State If $\|\nabla f(y_{t})\|\le \varepsilon$, then \textbf{terminate} with $\hat x = y_{t}$ and $\hat M = M_t$.\EndFor
			\EndFunction
		\end{algorithmic}
	\end{algorithm}
	
	A few remarks about Algorithm~\ref{alg:proposed_SCAR_noflag} are in order. First, whenever our guess $\mu_{t-1}$ is too large and the AR method is unable to reduce the gradient norm by half, we will simply discard the computed approximate solution $y_t$ and reduce our guess to $\mu_{t-1}/4$. Second, we need some initial guess $\mu_0$ and $M_0$ on the strong convexity and smoothness constants. We will show later that both can be simply set to $\|\nabla f(y_0) - \nabla f(z_0)\|/\|y_0 - z_0\|$ for any $z_0\not = y_0$. 
	The performance of Algorithm~\ref{alg:proposed_SCAR_noflag} is described below.
	
	\vgap
	\begin{theorem}
		\label{thm:conv_proposed_unconstrained_sc_noflag}
		Suppose that $\|\nabla f(y_0)\|> \varepsilon$. For any $\varepsilon>0$, $\mu_0\ge \mu$ and $M_0\le 4L$, Algorithm~\ref{alg:proposed_SCAR_noflag} terminates with $\|\nabla f(\hat x)\|\le\varepsilon$. The total number of gradient evaluations is bounded by  $(4+8\sqrt{5}C_1)\lceil\log_4(\mu_0/\mu)\rceil\sqrt{L/\mu}  + (4+8\sqrt{5}C_1)\sqrt{L/\mu}\lceil\log_2(\|\nabla f(y_0)\|/\varepsilon)\rceil$
		.
		\cmo{
		Moreover, if we have $\mu_0\le 2^{j}\varepsilon/\|y_0-x^*\|$ for some integer $j$, an alternative bound of gradient evaluations is $(4+8\sqrt{5}C_1)\lceil\log_4(\mu_0/\mu)\rceil\sqrt{L/\mu}  +(4+C_1\sqrt{10L/\mu})\max\{1,j\}$.}
	\end{theorem}
	\begin{proof}
		We start with three observations. First, by Theorem~\ref{thm:conv_proposed_unconstrained_searchL}, 
		the output $y_{t}$ of the AR method satisfies $\|\nabla f(y_{t})\|\le \mu_{t-1}\|y_{t-1}-x^*\|/2$ and the total number of gradient evaluations for computing $y_t$ is bounded by
		$4 +  C_1\sqrt{10L/\mu_{t-1}}$.
		Second, using the previous observation and 
		the relation 
  \begin{equation}\label{eq:growth}
  \|y_{t-1}-x^*\|\le \tfrac{1}{\mu}\|\nabla f(y_{t-1})\|
  \end{equation}
  due to the strong convexity of $f$, we conclude that $\|\nabla f(y_{t})\|\le \|\nabla f(y_{t-1})\|/2$  and hence that $\mu_t$ will no longer change whenever $\mu_{t-1}\le \mu$. Therefore $\mu_t\ge \mu/4$ for all $t$.
		Third, the algorithm will terminate with $\|\nabla f(y_t)\|\le \varepsilon$ whenever $t\ge P:=\lceil\log_4(\mu_0/\mu)\rceil + \lceil\log_2(\|\nabla f(y_0)\|/\varepsilon)\rceil$. Here $P$ is the sum of two parts: $\lceil\log_4(\mu_0/\mu)\rceil$ counts the number of times when $\|\nabla f(y_t)\|>\|\nabla f(y_{t-1})\|/2$ and the value of $\mu_t$ is reduced from $\mu_{t-1}$; $\lceil\log_2(\|\nabla f(y_0)\|/\varepsilon)\rceil$ counts the number of times when $\|\nabla f(y_t)\|\le\|\nabla f(y_{t-1})\|/2$. Summarizing the above observations, the total number of gradient evaluations $N$ is bounded by
		\begin{align*}
			N \le\ & 4P + C_1P\sqrt{40L/\mu} \le (4+8\sqrt{5}C_1)P\sqrt{L/\mu}
			\\
			\le\ & (4+8\sqrt{5}C_1)\lceil\log_4(\mu_0/\mu)\rceil\sqrt{L/\mu}  + (4+8\sqrt{5}C_1)\sqrt{L/\mu}\lceil\log_2(\|\nabla f(y_0)\|/\varepsilon)\rceil.
		\end{align*}
	
	\cmo{
		We may also analyze the total number of gradient evaluations from an alternative perspective when we know that  $\mu_0\le 2^{j}\varepsilon/\operatorname{dist}(y_0,X^*)$. By   Theorem~\ref{thm:conv_proposed_unconstrained_searchL}, the first AR call is guaranteed to yield $y_1$ such that $\|\nabla f(y_1)\| \le 2^{j-1}\varepsilon$. If $j\le 1$, then $y_1$ already satisfies $\|\nabla f(y_1)\|\le\varepsilon$. If $j>1$,
		then after $(j-1)$ occurrences of $\|\nabla f(y_t)\| \le \|\nabla f(y_{t-1})\|/2$ we will have $\|\nabla f(y_t)\|\le \varepsilon$. Therefore, the algorithm will terminate with $\|\nabla f(y_t)\|\le \varepsilon$ whenever $t\ge P:=\lceil\log_4(\mu_0/\mu)\rceil + j$ and we can bound $N$ by $N \le (4+8\sqrt{5}C_1)P\sqrt{L/\mu}$ similar as the proof of the previous bound.
	}	
	
	\qedsymbol
	\end{proof}
	
	\vgap 
	
	Based on the above theorem, our proposed SCAR algorithm is able to achieve that desired optimal complexity $\cO(1)\sqrt{L/\mu}\log(\|\nabla f(x_0)\|/\varepsilon)$ with a correct selection of $\mu$. With any initial guess $\mu_0$, SCAR is a parameter-free algorithm that  achieves the optimal complexity (up to the additive term $\cO(1)\log_4(\mu_0/\mu)\sqrt{L/\mu}$) for the first time in the literature.

 \begin{remark}
 Even though we focus on strongly convex problems in this section, our development also applies to a more general class of problems for which the condition in \eqref{eq:growth} holds. This type of condition is often called error bound or sharpness condition.
 \end{remark}
	
	\section{Minimizing gradient for unconstrained nonconvex problems}
	\label{sec:NCAR}
	In this section, we still consider the problem $f^*=\min_{x\in\R^n}f(x)$, but $f$ now is nonconvex with $l$-lower curvature. Throughout the section we will assume that the values of $L$ and $l$ are available\cmo{, $0\le l\le L$, }
	and that the optimal value $f^*$ is finite. 
	We will show that one can compute an approximate solution $\hat x$ with $\|\nabla f(\hat x)\|\le\varepsilon$ after at most $\cO(1)\sqrt{Ll}(\cmo{(f(x_0) - f^*)}/\varepsilon^2)$ gradient evaluations.  A parameter-free
	version of this method will be introduced in the next section since it is
	more complicated.
	
	A proximal point algorithm can be applied to minimize the gradient. Specifically, starting with $x^0\in\R^n$, let
	\begin{equation} \label{eq:def_nonconvex_ppt}
		x^i = \argmin_{x\in \R^n} \{F_i(x):= f(x) + l\|x - x^{i-1}\|^2\}.
	\end{equation}
	Since $F_i$ is strongly convex with modulus $l$ and $\nabla F_i(x^i)=0$, 
	\begin{align*}
		F_i(x_{i-1}) &\ge F_i(x_{i}) + \langle \nabla F_i(x_i), x_{i-1} - x_i\rangle + \frac{l}{2}\|x_{i-1} - x_i\|^2
		= F_i(x_{i}) + \frac{l}{2}\|x_{i-1} - x_i\|^2, 
	\end{align*}
	which, in view of the definition of $F_i$ and the relation $  2l(x^i-x^{i-1}) = - \nabla f(x^i)$ due to $\nabla F_i(x^i)=0$, then implies
	\begin{align}
		\label{eq:PPA_key_ineq}
		f(x^{i-1}) - f(x^i) \ge 3\|\nabla f(x^i)\|^2/(8l).
	\end{align}
	Summing from $i=1,\ldots,N$, and noting that $f(x^N) \ge f^*$, we conclude that 
	\begin{align}
		\label{eq:ppa_conv}
		\min_{i=1,\ldots, N}\|\nabla f(x^i)\|^2\le {({1}/{N})\sum_{i=1}^{N}\|\nabla f(x^i)\|^2} \le {{8l(f(x^0) - f^*)/(3N)}}.
	\end{align}
	Therefore, the number of subproblems of solving $x^i$ to compute a solution $\hat x$ with $\|\nabla f(\hat x)\|\le \varepsilon$ is bounded by $\cO(l(f(x^0) - f^*)/\varepsilon^2)$. However,  subproblem \eqref{eq:def_nonconvex_ppt} has to be solved exactly in the above analysis.
	
	\begin{algorithm}[h]
		\caption{\label{alg:proposed_unconstrained_nc} The nonconvex acceleration through strongly convex accumulative regularization for gradient minimization}
		\begin{algorithmic}
			\Require Initial iterate \cmo{ $y^0$}, lower curvature constant $l$ and Lipschitz constant $L$.
			\State 
			\cmo{Denote $F_0(x):=f(x) + \tilde l \| x- y^0\|^2$ where $\tilde l\in [l, L]$ is a parameter. Compute $x^0 = \text{AR}(F_0, \varepsilon/4, y^{0}, \tilde l, L+2\tilde l)$. If $\|\nabla f(x^0)\|\le \varepsilon$, then \textbf{terminate} with $\hat x:=x^{0}$ .}
			\For{$i=1,2,\ldots, $}
			\State
			$	x^{i} = \text{SCAR}(\cmo{F_i}, \varepsilon/4, x^{i-1}, l, L+2 l)$
			 \cmo{(see description of SCAR in Algorithm~\ref{alg:proposed_SCAR_noflag} and definition of $F_i$ in \eqref{eq:def_nonconvex_ppt}).}
			\State If $\|\nabla f(x^i)\|\le \varepsilon$, then \textbf{terminate} with $\hat x:=x^{i}$ .
			\EndFor
		\end{algorithmic}
	\end{algorithm}
	
	In this section, we suggest to use the SCAR algorithm in the previous section to solve subproblem \eqref{eq:def_nonconvex_ppt} approximately. 
	Our algorithm is described in Algorithm~\ref{alg:proposed_unconstrained_nc}. 
	Note that when we call Algorithm~\ref{alg:proposed_SCAR_noflag}
	to solve \eqref{eq:def_nonconvex_ppt}, accurate estimation for the smoothness and strongly
	convex constants are provided, and the SCAR algorithm simply restarts the basic AR method in Algorithm~\ref{alg:proposed_unconstrained}.
	Also, different from the aforementioned vanilla proximal point algorithm with $\nabla F^{(i)}(x^i)=0$, in Algorithm~\ref{alg:proposed_unconstrained_nc} we  
	compute $x^i$ such that $\|\nabla F^{(i)}(x^i)\|\le \varepsilon/4$. We show that such $\varepsilon/4$ accuracy threshold allows us to establish a result that is analogous to \eqref{eq:PPA_key_ineq}.

	\vgap
	
	\begin{lemma}
		\label{lem:SCAR_subproblem}
		Suppose that $f$ is $L$-smooth. Let $l\in (0,L]$ be the lower curvature constant of $f$. For a fixed $u\in\R^n$, consider the strongly convex problem
		\begin{align}
			\bar x^*:=\argmin_{y\in\R^n}\{F(y):=f(y) + \tilde l\|y - u\|^2\},
		\end{align}
		where $\tilde l\ge l$. 
		If $x$ satisfy $\|\nabla f(x)\|>\varepsilon$ and $\|\nabla F(x)\|\le \varepsilon/4$, then we have 
		$\|\nabla f(x)\|^2 \le 10\tilde l(f(u) - f(x)).
		$
	\end{lemma}
	\begin{proof}
		By $\nabla F(x) = \nabla f(x) + 2\tilde l(x-u)$,
		$\|\nabla f(x)\|>\varepsilon$ and $\|\nabla F(x)\|\le \varepsilon/4$, 
		\begin{align}
			\label{tmp:eq:gradF_bound}
			2\tilde l\|x -u \|\ge \|\nabla f(x)\| - \|\nabla F(x)\|>3\varepsilon/4\ge 3\|\nabla F(x)\|.
		\end{align}
		As a direct consequence of the above observation, we have
		\begin{align}	\label{tmp:eq:gradf_bound}
			\|\nabla f(x)\| = \|\nabla F(x) - 2\tilde l(x-u)\| \le \|\nabla F(x)\| + 2\tilde l\|x - u\| \le (8\tilde l/3)\|x-u\|.
		\end{align}
		Also, it follows from $F(u) \ge F(\bar x^*)$,
		the strong convexity of $F$ (modulus $2\tilde l-l$), and
		the bound on $\|\nabla F(x)\|$ in \eqref{tmp:eq:gradF_bound} that 
		$$F(x) - F(u)\le F(x) - F(x^*) \le \|\nabla F(x)\|^2/(2(2\tilde l - l))<(2\tilde l/9)\|x - u\|^2
		, $$
		or equivalently, $f(x) - f(u) \le -(7\tilde l/9)\|x - u\|^2$. Combining this result and \eqref{tmp:eq:gradf_bound} we have
		$\|\nabla f(x)\|^2 \le 10\tilde l(f(u) - f(x)).$ 
 
	\qedsymbol\end{proof}
	
	\vgap
	
	Applying Lemma~\ref{lem:SCAR_subproblem} to subproblem \eqref{eq:def_nonconvex_ppt}, we have $\|\nabla f(x^i)\|^2 \le 10\cmo{\tilde l}(f(x^{i-1}) - f(x^i))$, a result that is analogous to \eqref{eq:PPA_key_ineq} in our previous analysis of the vanilla proximal point algorithm. Therefore, we are now able to obtain a convergence result that is similar to \eqref{eq:ppa_conv}. It remains to estimate the total number of gradient evaluations. We need the following technical lemma for the analysis.
	
	\vgap
	\begin{lemma}
		\label{lem:tech_nc}
		For any $\Delta>0$, $\varepsilon>0$, and any \cmo{vector $c=(c_1,\ldots,c_N)\in\R^N$ such that $c_i>0$ for all $i=1,\ldots,N$}, if the optimization problem
		\begin{align} \label{eq:aux_problem}
			H(\Delta,\varepsilon, c):=& \max_{y_1,\ldots,y_N\in\R} 
			\left\{\sum_{i=1}^{N}\frac{1}{c_i}\log_2\frac{y_i}{\varepsilon}
			: \ \sum_{i=1}^{N}\frac{y_i^2}{c_i^2}\le \Delta;\ y_i\ge \varepsilon,\ \forall i\right\},
		\end{align}
		is feasible, then $H(\Delta,\varepsilon, c)\le \|c\|_\infty\Delta/\varepsilon^2$.
	\end{lemma}
	\begin{proof}
		By Lagrange duality, we have $H(\Delta,\varepsilon, c)\le d(\lambda)/(\ln 2)$ for all $\lambda>0$, where
		\begin{align*}
			d(\lambda):=& \max_{y_1,\ldots,y_N\in\R} \sum_{i=1}^{N}\frac{1}{c_i}\ln\frac{y_i}{\varepsilon} + \lambda\left(\Delta - \sum_{i=1}^{N}\frac{y_i^2}{c_i^2}\right)
			\ \st \ y_i\ge \varepsilon,\ \forall i=1,\ldots,N.
		\end{align*}
		Note that the above Lagrange relaxation problem has an optimal solution $y_i^*(\lambda) =  \max\{\varepsilon, \sqrt{c_i/(2\lambda)}\}$ for all $i=1,\ldots,N$.
		Observe that the feasibility of \eqref{eq:aux_problem} implies that $\sum_{i=1}^{N}\varepsilon^2/c_i^2\le \Delta$.
		This observation, by the intermediate value theorem (by tending $\lambda$ to $0$ and $+\infty$ in $\sum_{i=1}^{N}(y_i^*(\lambda)/c_i)^2$ ), implies that there exists $\overline{\lambda}>0$ s.t.
		$\sum_{i=1}^{N}(y_i^*(\overline{\lambda})/c_i)^2 = \Delta$. Using the fact that $\ln t\cmo{<} (1/2)t^2$ for all $t>0$, we have
		\begin{align*}
			H(\Delta,\varepsilon,c)\le \frac{d(\overline{\lambda})}{\ln 2} \le  \frac{\|c\|_\infty}{\ln 2}\sum_{i=1}^{N}\frac{1}{c_i^2}\ln\frac{y_i^*(\overline{\lambda})}{\varepsilon} \le \frac{\|c\|_\infty}{2\ln 2}\sum_{i=1}^{N}\left(\frac{y_i^*(\overline{\lambda})}{c_i\varepsilon}\right)^2 \le \frac{\|c\|_\infty\Delta}{\varepsilon^2}.
		\end{align*}
  
	\qedsymbol\end{proof}
	
	\vgap
	
	We are now ready to present the convergence properties of Algorithm~\ref{alg:proposed_unconstrained_nc}.
	
	\vgap
	
	\begin{theorem}
		\label{thm:conv_nc}
		Suppose that $\|\nabla f(\cmo{y}^0)\|>\varepsilon$. 
		Algorithm~\ref{alg:proposed_unconstrained_nc} terminates with an solution $\hat x$ such that $\|\nabla f(\hat x)\|\le\varepsilon$. The total number of gradient evaluations required by the algorithm is bounded by \cmo{$4 + 3\sqrt{\max\{(L+2\tilde l)/2, 2L\}/\sigma_1}+16\sqrt{2c_\cA }\sqrt{L/\tilde l} + 30\sqrt{3}(4+16\sqrt{5}C_1)\sqrt{L l}(f(y^0) - f(x^*))/\varepsilon^2$}, where $C_1$ is a universal constant defined in Theorem~\ref{thm:conv_proposed_unconstrained_searchL}.
		\cmo{
		Moreover, if $\tilde l = \max\{l, \varepsilon/\operatorname{dist}(y^0, X^*)\}$, then the total number of gradient evaluations is bounded by $4 + (3\sqrt{2} + 16\sqrt{2c_\cA })\sqrt{L\operatorname{dist}(y^0, X^*) / \varepsilon} + 30\sqrt{3}(4+16\sqrt{5}C_1)\sqrt{L l}(f(y^0) - f(x^*))/\varepsilon^2$.
  	}
		
	\end{theorem}
	\begin{proof}
		\cmo{ Let us start with the assumption that $\|\nabla f(x^0)\|>\varepsilon$ after $x^0$ is computed from the AR call. }
		Applying Lemma~\ref{lem:SCAR_subproblem} to subproblem \eqref{eq:def_nonconvex_ppt}, we have the result $\|\nabla f(x^i)\|^2 \le 10l(f(x^{i-1}) - f(x^i))$ for all $i$. Taking telescopic sum of these relations, we have
		\begin{align}
			\label{eq:nc_contradiction}
			N \min_{i=\cmo{ 0,\ldots,N-1}}\|\nabla f(x^i)\|^2 \le \sum_{i=\cmo{0}}^{\cmo{N-1}}\|\nabla f(x^i)\|^2 \le {10{l}(f(\cmo{y}^0) - f(x^*))}.
		\end{align}
		Consequently, Algorithm~\ref{alg:proposed_unconstrained_nc}
		is guaranteed to terminate within $\bar N := \lceil 10{l}(f(\cmo{y}^0) - f(x^*))/\varepsilon^2\rceil$ calls to the SCAR function. 
		It remains to estimate the total number of gradient evaluations.  Let $N(\varepsilon)\le \bar N$ be the smallest $i$ such that $\|\nabla f(x^i)\|\le \varepsilon$. Clearly by \eqref{eq:nc_contradiction}, $\sum_{i=\cmo{1}}^{N(\varepsilon)}\|\nabla f(x^{i-1})\|^2\le 10 l (f(\cmo{y}^0) -f(x^*))$. Applying Lemma~\ref{lem:tech_nc} (with $\Delta = 10 l(f(\cmo{y}^0) -f(x^*))$, $N = N(\varepsilon)-1$, and $c_i\equiv 1$), we have
		\begin{align}
			\label{eq:tmp1}
			\sum_{i=\cmo{0}}^{N(\varepsilon)-1}\log_2(\|\nabla f(x^{i})\|/\varepsilon) \le 10 l (f(\cmo{y}^0) - f(x^*))/\varepsilon^2.
		\end{align}
			Using the above relation and 		Theorem~\ref{thm:conv_proposed_unconstrained_sc_noflag} (with $M_0=L+2{l}$ and $\mu_0 = \mu = {l}$), 
	we can bound the total number of gradient evaluations of $\nabla f$ by 
	
	\begin{align*}
		& \sum_{i=1}^{N(\varepsilon)} (4+8\sqrt{5}C_1)\sqrt{(L+2{l})/{ l}}\log_2(4\|\nabla f(x^{i-1})\|/\varepsilon)  
		\le 
		30\sqrt{3}(4+8\sqrt{5}C_1)\sqrt{L l}(f(\cmo{y}^0) - f(x^*))/\varepsilon^2.
	\end{align*}

		\cmo{
		To finish the proof it suffices to estimate the total number of evaluations in the AR call for computing $x^0$. Such estimate will also resolve the case when $\|\nabla f(x^0)\|\le\varepsilon$.
		Applying Theorem \ref{thm:conv_proposed_unconstrained_searchL} such number is bounded by 
		 $4 + 3\sqrt{\max\{(L+\tilde 2l)/2, 2L\}/\tilde l}+16\sqrt{2c_\cA }\sqrt{L/\tilde l}$. 
		Moreover, note that if $\tilde l = \max\{l, \varepsilon/\operatorname{dist}(y^0, X^*)\}$, then such number is further bounded by $4 + (3\sqrt{2} + 16\sqrt{2c_\cA })\sqrt{L\operatorname{dist}(y^0, X^*) / \varepsilon}$.
		}
	\qedsymbol\end{proof}
	
	\vgap
	
	In the above result we obtain $\cO(1)\sqrt{Ll}(f(\cmo{y}^0) - f(x^*))/\varepsilon^2$ complexity for nonconvex gradient minimization with $l$-lower curvature constant.
	To the best of our knowledge, this is the first time such complexity is achieved in the literature. 
	
	\section{Parameter-free algorithm for \cmo{ smooth} problems}
	\label{sec:NCAR_PF}
	Algorithm~\ref{alg:proposed_unconstrained_nc} requires the knowledge of $l$ and $L$. While the requirement of $L$ may be relaxed easily through backtracking,  the knowledge of $l$ is critical. In particular, if we supply Algorithm~\ref{alg:proposed_unconstrained_nc} with an estimate $\tilde l<l$, 
	we may not be able to even guarantee the convexity of $F_i(x) :=f(x) + \tilde l \|x - x^{i-1}\|^2$. 
	In this section, we will first present a variant of the strongly convex accumulative regularization algorithm, referred to as SCAR-PM, that can handle plausible estimate of the strong convexity modulus of a smooth function. 
	Given any smooth function $f$ and a guess $\tilde\mu$ of the strong convexity modulus, SCAR-PM will either compute an approximate solution $\hat x$ with $\|\nabla f(\hat x)\|\le \varepsilon$, or return an error that $f$ is not $\tilde\mu$-strongly convex. With the help of SCAR-PM, we will then develop a nonconvex acceleration algorithm for solving the gradient minimization problem without any problem parameters. \cmo{ Our proposed algorithm is applicable to strongly convex, convex, and nonconvex problems without requiring any knowledge of the problem's convexity status.}
	
	\subsection{Handling plausible strong convexity information}
	The SCAR method in Algorithm~\ref{alg:proposed_SCAR_noflag} relies on restarting our AR algorithm for gradient minimization of convex smooth functions. The AR algorithm then calls subroutine $\cA$ for solving its subproblems. In order to handle plausible strong convexity information and possibly nonconvex objective function, we need to modify the assumption on the subroutine $\cA$. Specifically, we will assume the following convergence properties of $\cA$:
	
	\vgap
	\begin{assumption}
		\label{assum:A_unconstrained_searchL_pm}
		The subroutine $(x_s, L_s) =\cA(f, \sigma_s, \xu_s, x_{s-1})$ is able to produce an estimate $L_s^k$ of Lipschitz constant $L$ such that $L_s^k\le c_\cA L$ after its $k$-th gradient evaluation of $\nabla f$, where $c_\cA$ is a universal constant that depends on algorithm $\cA$. Moreover, if $f$ is convex, then $\cA$ has performance guarantee \eqref{eq:A_cond_unconstrained_searchL} after the $k$-th gradient evaluation of $\nabla f$ for the computed approximate solution $x_s^k$. 
		Subroutine $\cA$ will terminate whenever $k\ge 8\sqrt{2L_s^k/\sigma_s}$ and output the approximate solution $x_s = x_s^k$.
	\end{assumption}
	\vgap
	
	The main difference between the above assumption and the previous Assumption~\ref{assum:A_unconstrained_searchL} is that we now allow the input function $f$ of subroutine $\cA$ to be nonconvex. Note that we do not require $\cA$ to find out whether $f$ is convex.  We only need to ensure the convergence properties in \eqref{eq:A_cond_unconstrained_searchL} hold if $f$ is convex. For nonconvex $f$, we do not have any requirement on performance guarantees except that an estimate $L_s^k$ of Lipschitz constant can be computed to verify the termination criterion. It should be noted that the potential nonconvexity of $f$ will not impact the termination of either subroutine $\cA$ or the AR algorithm, because $\cA$ will terminate based on the number of iterations (i.e., whenever $k\ge 8\sqrt{2L_s^k/\sigma_s}$) and AR will  terminate whenever $\sigma_s\ge M_s$. Since $L_s^k$ and $M_s$ are both lower estimates of $L$, both termination conditions will be satisfied after a finite number of iterations.

	\begin{algorithm}[h]
		\caption{\label{alg:proposed_SCAR} The  strongly convex accumulative regularization algorithm with plausible strong convexity modulus (SCAR-PM)}
		\begin{algorithmic}
			\Function{$(\hat x, \hat M, \text{ERROR})$=SCAR-PM}{$f$, $\varepsilon$, $y_0$, $\tilde\mu$, $M_0$}
			\For {$t=1,2,\ldots$}
			\State Set 
			$(y_{t},M_{t}) = \text{AR}(f, y_{t-1}, \tilde\mu/10, M_{t-1})$ (see Algorithm~\ref{alg:proposed_unconstrained_searchL}). 
			\State If $\|\nabla f(y_{t})\|\le \varepsilon$, then \textbf{terminate} with $\hat x = y_{t}$, $\hat M=M_t$ and ERROR$=$FALSE.
			\State If {$\|\nabla f(y_t)\|>\|\nabla f(y_{t-1})\|/2$}, then \textbf{terminate} with $\hat x = y_0$, $\hat M = M_t$, and ERROR$=$TRUE.
			\EndFor
			\EndFunction
		\end{algorithmic}
	\end{algorithm}
	
	We are now ready to describe the variant of strongly convex accumulative regularization algorithm with plausible strong convexity modulus (SCAR-PM).
	With a supplied guess $\tilde\mu$ of the strong convexity modulus, SCAR-PM will either return ERROR$=$FALSE with a solution $\hat x$ such that $\|\nabla f(\hat x)\|\le \varepsilon$, or throw an error by returning ERROR$=$TRUE. The latter case is discovered whenever the AR method terminates but the gradient norm of $f$ is not reduced by half comparing with the previously computed gradient norm. This could happen if $f$ is $\mu$-strongly convex with  $\mu\le \tilde \mu$, convex but not strongly convex, or nonconvex. The following theorem describes the performance of the SCAR-PM method.
	
	\vgap
	\begin{theorem}
		\label{thm:conv_proposed_unconstrained_sc}
		Assume that $\|\nabla f(y_0)\|>\varepsilon$ and $M_0\le 4L$.
		Algorithm~\ref{alg:proposed_SCAR} will either find an approximate solution $\hat x$ with $\|\nabla f(\hat x)\|\le\varepsilon$ or report that $f$ is not $\tilde\mu$-strongly convex by returning ERROR$=$TRUE. It terminates after at most $(4+C_1\sqrt{10L/\tilde\mu})\lceil\log_2(\|\nabla f(y_0)\|/\varepsilon)\rceil$ gradient evaluations, where $C_1$ is a universal constant defined in Theorem~\ref{thm:conv_proposed_unconstrained_searchL}. 
		\cmo{ Moreover, if $\tilde\mu\le  2^{j}\varepsilon/\operatorname{dist}(y_0,X^*)$ for some integer $j$, an alternative bound of gradient evaluations is $(4+C_1\sqrt{10L/\tilde\mu})\max\{1,j\}$.}
	\end{theorem}
	\begin{proof}
		Observe that under Assumption~\ref{assum:A_unconstrained_searchL_pm}, the gradient evaluation complexity analysis of Theorem~\ref{thm:conv_proposed_unconstrained_searchL} do not rely on the convexity of $f$; it remains the same for nonconvex $f$. Therefore, the total number of gradient evaluations for computing $y_t$ is bounded by
		$4 + C_1\sqrt{10L/\tilde \mu}$, where $C_1:= \sqrt{2}(3 + 16\sqrt{2c_\cA})$. If $f$ is $\mu$-strongly convex and $\tilde\mu\le \mu$, then by Theorem~\ref{thm:conv_proposed_unconstrained_searchL} we have $\|\nabla f(y_t)\|\le \tilde\mu \operatorname{dist}(\cmo{y_{t-1}}, X^*)/2\le \|\nabla f(y_{t-1})\|/2$. Equivalently, whenever the condition is violated, i.e., $\|\nabla f(y_t)\|>\|\nabla f(y_{t-1})\|/2$, the SCAR-PM method will correctly report that $f$ is not $\tilde\mu$-strongly convex by terminating and returning ERROR$=$TRUE.  
		Moreover, whenever $t\ge T:=\lceil\log_2(\|\nabla f(y_0)\|/\varepsilon)\rceil$, we will have $\|\nabla f(y_t)\|\le\varepsilon$ and the SCAR-PM method will terminate with ERROR$=$FALSE. The total number of gradient evaluations $N$ can now be bounded by 
		\begin{align*}
			N\le\ & 4T + C_1\sqrt{10}T\sqrt{L/\tilde \mu} 
			\\
			= & (4+C_1\sqrt{10L/\tilde\mu})\lceil\log_2(\|\nabla f(y_0)\|/\varepsilon)\rceil
			.
		\end{align*}
		Note that the above bound on total number of gradient evaluations is valid regardless of whether the returned ERROR is either TRUE or FALSE. 

        \cmo{
        We may also analyze the total number of gradient evaluations from an alternative perspective when we know that  $\tilde\mu\le 2^{j}\varepsilon/\operatorname{dist}(y_0,X^*)$. By   Theorem~\ref{thm:conv_proposed_unconstrained_searchL}, the first AR call is guaranteed to yield $y_1$ such that $\|\nabla f(y_1)\| \le 2^{j-1}\varepsilon$. If $j\le 1$, then $y_1$ already satisfies $\|\nabla f(y_1)\|\le\varepsilon$. If $j>1$ and SCAR-PM terminates with ERROR$=$FALSE, we have $\|\nabla f(y_t)\| \le \|\nabla (y_{t-1})\|/2$ and hence $\|\nabla f(y_t)\|\le\varepsilon$ whenever $t\ge j$. Therefore, the total number of gradient evaluations $N$ can now be bounded by $
        	N\le\  (4+C_1\sqrt{10L/\tilde\mu})\max\{1,j\}$.
        Similar as the proof of the previous bound, the above bound on total number of gradient evaluations is valid regardless of whether the returned ERROR is either TRUE or FALSE. 
        }
	\qedsymbol\end{proof}

\vgap

	\vgap
	
	\subsection{The NASCAR algorithm for \cmo{ smooth} problems}
	With the help of SCAR-PM in the previous subsection, we are now ready to design a parameter-free nonconvex acceleration through strongly convex accumulative regularization (NASCAR) method as shown in Algorithm~\ref{alg:proposed_unconstrained_nc_searchmu}.

	Each iteration of NASCAR computes a solution $x^i$ s.t. $\|\nabla F^{i}(x^i)\|\le \varepsilon/4$, where $F^{(i)}(x):=f(x) + l_i\|x - x^{i-1}\|^2$ and $l_i$ is our estimate of the lower curvature constant $l$. Note that $F^{(i)}$ is strongly convex if $l_i\ge l$, but our guess of $l_i$ might be incorrect with $l_i< l$ and thus $F^{(i)}$ might not be convex. In the latter case, the SCAR-PM algorithm used in NASCAR would throw an error with ERROR$=$TRUE after discovering the supplied strong convexity constant is wrong, indicating that our guess $l_i$ is incorrect with $l_i<l$. However, it should be noted that the SCAR-PM algorithm may produce $x^i$ with $\|\nabla F^{i}(x^i)\|\le \varepsilon/4$ successfully before discovering that the supplied strong convexity modolus is wrong. In such case, we could also confirm that $l_i<l$ through Lemma~\ref{lem:SCAR_subproblem}. Specifically, by Lemma~\ref{lem:SCAR_subproblem}, $\|\nabla F^{(i)}(x^{i-1})\|\le\varepsilon/4$ implies that $\|\nabla f(x^i)\|^2\le 10l(f(x^{i-1}) - f(x^i))$ when $l_i\ge l$. Equivalently, whenever $\|\nabla F^{(i)}(x^i)\|\le\varepsilon/4$ but $\|\nabla f(x^i)\|^2> 10\tilde l(f(x^{i-1}) - f(x^i))$, we must have our guess $l_i<l$. Therefore, NASCAR could discover that $l_i<l$ and increase it whenever either ERROR$=$TRUE or $\|\nabla f(x^i)\|^2\le 10l_i(f(x^{i-1}) - f(x^i))$. 
	In Theorem~\ref{thm:NASCAR_conv}, we state the performance of the NASCAR algorithm.

	\begin{algorithm}[h]
		\caption{\label{alg:proposed_unconstrained_nc_searchmu} Nonconvex acceleration through strongly convex accumulative regularization (NASCAR) for gradient minimization}
		\begin{algorithmic}
			\Function{$\hat x =\ $NASCAR}{$x^0$, $\varepsilon$, $l_0$, $M_0$}
			\For{$i=1,\ldots, $}
			\State Set $F^{(i)}(x):=f(x) + l_{i-1}\|x - x^{i-1}\|^2$.
			\State Set ($x^{i}$, $M_{i}$, ERROR$_i$) = SCAR-PM($F^{(i)}$, $\varepsilon/4$, $x^{i-1}$, $l_{i-1}$, $M_{i-1}$).
			(see  Algorithm~\ref{alg:proposed_SCAR}).
			\State If $\|\nabla f(x^i)\|\le \varepsilon$, then \textbf{terminate} with $\hat x:=x^{i}$. 
			\State If {ERROR$_i=$TRUE} or \cmo{ $\|\nabla f(x^{i})\|^2 > 10l_{i-1}(f(x^{i-1}) - f(x^{i}))$, then set $l_i=4*l_{i-1}$ and $x^i=x^{i-1}$ respectively.}
			\EndFor			
			\EndFunction
		\end{algorithmic}
	\end{algorithm}

	\vgap
	
	\begin{theorem}
		\label{thm:NASCAR_conv}
		Supppose that the input arguments of NASCAR satisfy $\|\nabla f(x^0)\|> \varepsilon$, $l_0\le l$ and $M_0\le 4L$. 
		The NASCAR algorithm will output an approximate solution $\hat x$ such that $\|\nabla f(\hat x)\|\le\varepsilon$. The total number of gradient evaluations of $\nabla f$ required by the NASCAR algorithm is bounded by $8\sqrt{10}C_1\sqrt{L/M_{0}}\log_2(4\|\nabla f(x^{0})\|/\varepsilon) + 160\sqrt{10}C_1\sqrt{Ll}(f(x^0) - f(x^*))/\varepsilon^2$, where $C_1$ is the universal constant described in Theorem~\ref{thm:conv_proposed_unconstrained_searchL}.
	\end{theorem}
	
	\begin{proof}
		Note that when $l_i\ge l$, then SCAR-PM will always return $ERROR_i=FALSE$ with $\|\nabla f(x^i)\|^2\le 10l_i(f(x^{i-1}) - f(x^i))$. Therefore we have $l_i\le 4l$ for all $i\ge 1$.
		Applying Theorem~\ref{thm:conv_proposed_unconstrained_sc} to the $i$-th call to the SCAR-PM function in the NASCAR method, we have that the SCAR-PM function terminates after at most $(4 + C_1\sqrt{\max\{5M_{i-1}, 20L\}/l_{i-1}})\log_2(4\|\nabla f(x^{i-1})\|/\varepsilon)$ gradient evalutions, and that the output $M_i\le \max\{M_{i-1}/2, 2(L+2l_{i-1})\}$. Noting that $M_0\le L$ and recalling that $l_i\le 4l\le 4L$ for all $i\ge 1$ and $C_1\ge 3$, we have $M_i\le 18L$ and that the number of gradient evaluations required by the $i$-th call to the SCAR-PM function is bounded by 
		$$(4 + C_1\sqrt{90L/l_{i-1}})\log_2(4\|\nabla f(x^{i-1})\|/\varepsilon)\le 4\sqrt{10} C_1\sqrt{L/l_{i-1}}\log_2(4\|\nabla f(x^{i-1})\|/\varepsilon).$$ 	
		The total number of SCAR-PM calls is bounded by $N:=N_1 + N_2$, where $N_1:=\lceil\log_4{4l/l_0}\rceil$ counts the total number of times we need to increase $l_i$ and $N_2:=\lceil 40l(f(x_0) - f(x^*))/\varepsilon^2\rceil$ counts the total number of times when $\|\nabla f(x^i)\|^2\le 10l_i(f(x^{i-1}) - f(x^i))$. 
		
		Let us assume that the NASCAR method terminates at $i=N(\varepsilon)\le N$. Among the indices $\{1,\ldots,N(\varepsilon)\}$, let $I:=\{i_k\}_{k=1}^{K}$ (where $K\le N_2$ and $i_K = N(\varepsilon)$) be the indices at which the SCAR-PM call returns ERROR$_{i_k}=$FALSE and $\|\nabla f(x^{i_k})\|^2\le 10l_{i_k}(f(x^{i_k-1}) - f(x^{i_k}))$, and let $\bar{I}$ be the rest of the indices when $l_i$ is replaced by $4l_i$. Noting that for all $i\in\bar I$ we always set $x^i = x^{i-1}$, we conclude that $\|\nabla f(x^{i_k})\|^2\le 10l_{i_k}(f(x^{i_{k-1}}) - f(x^{i_k}))$ for all $k\ge 1$. Here for convenience we set $i_0 = 0$. For all $i\in\bar I$, we have $l_i = 4l_{i-1}$. Specially, the total number of gradient evaluations required by all the $i$-th SCAR-PM calls with $i\in (i_{k-1},i_k]$ is bounded by
		\begin{align*}
			& \sum_{i=i_{k-1}+1}^{i_k}4\sqrt{10} C_1\sqrt{L/l_{i-1}}\log_2(4\|\nabla f(x^{i-1})\|/\varepsilon)
			\\
			\le & 8\sqrt{10}C_1\sqrt{L/l_{i_{k-1}}}\log_2(4\|\nabla f(x^{i_{k-1}})\|/\varepsilon).
		\end{align*}
		Therefore the total number of gradient evaluations required by the NASCAR method is bounded by
		\begin{align*}
			& \sum_{k=1}^{K}8\sqrt{10}C_1\sqrt{L/l_{i_{k-1}}}\log_2(4\|\nabla f(x^{i_{k-1}})\|/\varepsilon)
			\\
			= & 8\sqrt{10}C_1\sqrt{L/M_{0}}\log_2(4\|\nabla f(x^{0})\|/\varepsilon) + \sum_{k=1}^{K-1}8\sqrt{10}C_1\sqrt{L/l_{i_{k}}}\log_2(4\|\nabla f(x^{i_{k}})\|/\varepsilon).
		\end{align*}
		Here recalling that $\|\nabla f(x^{i_k})\|^2\le 10l_{i_k}(f(x^{i_{k-1}}) - f(x^{i_k}))$, by  Lemma~\ref{lem:tech_nc} (with $4\varepsilon$ instead of $\varepsilon$, $y_i=4\|\nabla f(x^{i_k})\|$, $\Delta = 160(f(x^0) - f(x^*))$, $N=K-1$, and $c_i = \sqrt{l_{i_k}}$) we have
		\begin{align*}
			\sum_{k=1}^{K-1}(1/\sqrt{l_{i_k}})\ln(\|\nabla f(x^{i_k})\|/\varepsilon) \le 20\sqrt{l}(f(x^0) - f(x^*))/\varepsilon^2.
		\end{align*}
		Thus the total number of gradient evaluations required by the NASCAR method is bounded by $8\sqrt{10}C_1\sqrt{L/M_{0}}\log_2(4\|\nabla f(x^{0})\|/\varepsilon) + 160\sqrt{10}C_1\sqrt{Ll}(f(x^0) - f(x^*))/\varepsilon^2$.
	\qedsymbol\end{proof}
	
	\vgap
	
	In the above theorem, we need input arguments to satisfy $l_0\le l$ and $M_0\le 4L$. We can compute such arguments by an initialization procedure described in Algorithm \ref{alg:proposed_NASCAR_init}.
	
	\begin{algorithm}[h]
		\caption{\label{alg:proposed_NASCAR_init}$(l_0, M_0, \hat x)$ = NASCAR-Init($x^0$, $\varepsilon$)}
		\begin{algorithmic}
			\State Set $M_0 = \|\nabla f(y_0) - \nabla f(z_0)\|/\|y_0 - z_0\|$ for any $z_0\not = y_0$.
			\State Set $\tilde l_0 = M_0$.
			\For{$i=1,\ldots,$}
				\State Set $\tilde F^{(i)}(x):=f(x) + \tilde l_{i-1}\|x - x^0\|^2$.
				\State \cmo{ Set $\hat x^i$ = SCAR-PM($f$, $\varepsilon$, $x^0$, $\tilde l_{i-1}$, $M_0$).}
				\State \cmo{ If $\|\nabla f(\hat x^i)\|\le\varepsilon$, then \textbf{terminate} and report the approximate solution $\hat x^i$. }
				\State Set $(\tilde x^i, \tilde M_i, $ ERROR$_i$) = SCAR-PM($\tilde F^{(i)}$, $\varepsilon/4$, $x^0$, $\tilde l_{i-1}$, $M_0$)
				\State If $\|\nabla f(\tilde x^i)\|\le\varepsilon$, then \textbf{terminate} and report the approximate solution $\tilde x^i$. 
				\State If ERROR$_i$~=FALSE or $\|\nabla f(\tilde x^i)\|^2>10\tilde l_{i-1}(f(x^0) - f(\tilde x^i))$, then \textbf{terminate} with $l_0 = \tilde l_{i-1}$. Otherwise, set $\tilde l_i = \tilde l_{i-1}/4$.
			\EndFor
		\end{algorithmic}
	\end{algorithm} 

	\cmo{ When solving nonconvex problems, s}imilar to the NASCAR algorithm, the above NASCAR-Init procedure keeps calling the SCAR-PM function to compute approximate solution $\tilde x^i$ with $\|\nabla \tilde F^{(i)}(\tilde x^i)\|\le\varepsilon/4$. One major difference is that the NASCAR-Init is looking for a value $\tilde l_{i-1}$ with which we can confirm that the function $\tilde F^{(i)}$ is not $\tilde l_{i-1}$-strongly convex. Such value $\tilde l_{i-1}$ is guaranteed to satisfy $\tilde l_{i-1}\le l$. Note that it is also possible for us to compute an approximate solution \cmo{ $\hat x=\tilde x^i$ or $\hat x^i$ with the desired accuracy $\|\nabla f(\hat x)\|\le\varepsilon$} in the NASCAR-Init procedure. \cmo{ Such approximate solution $x$ can be computed for strongly convex, convex, and nonconvex problems without the knowledge of the problem's convexity status.}
	In the following proposition, we describe the total number of gradient evaluations required by NASCAR-Init.
	
	\vgap
	\begin{pro}
		\label{pro:NASCAR-init}
		Suppose that $\|\nabla f(x^0)\|>\varepsilon$. The NASCAR-Init procedure will terminate with either an approximate solution $\hat x$ such that $\|\nabla f(\hat x)\|\le \varepsilon$, or input arguments $l_0\le l$ and $M_0\le 4L$ that can be used in the NASCAR algorithm. Moreover, when \cmo{ the function $f$ is strongly convex, convex, or nonconvex with lower curvature constant} $l\le \varepsilon^2/(10(f(x^0) - f(x^*)))$, NASCAR-Init is guaranteed to terminate with $\|\nabla f(\hat x)\|\le \varepsilon$.
		The total number of gradient evaluations of $\nabla f$ required by the NASCAR-Init procedure is bounded by
		\cmo{
		\begin{itemize}
			\item $8\lceil\log_{4} (M_0/\mu) \rceil + 4C_1\sqrt{30L/\mu}\log_2(4\|\nabla f(x^0)\|/\varepsilon)$ for strongly convex $f$;
			\item $64 C_1\sqrt{30L\operatorname{dist}(x_0,X^*)/\varepsilon}$
			for convex $f$; and
			\item $4C_1\max\{\sqrt{30L/M_0}, 10\sqrt{2L(f(x^0) - f(x^*))}/\varepsilon\}\lceil\log_2(4\|\nabla f(x^0)\|/\varepsilon)\rceil$\\$+8\lceil\max\{1, 2\log_4(\sqrt{40L(f(x^0) - f(x^*))}/\varepsilon)\}\rceil
			$
			for nonconvex $f$.
		\end{itemize} 
		H}ere $C_1$ is the universal constant described in Theorem~\ref{thm:conv_proposed_unconstrained_searchL}.
	\end{pro}
	\begin{proof}
		\cmo{ We start by showing that the NASCAR-Init procedure will terminate for strongly convex, convex, and nonconvex problems. In the strongly convex case, observing that $\mu\le M_0\le L$ and that the function call SCAR-PM$(f,\varepsilon,x^0,\tilde{l}_{i-1}, M_0)$ will terminate with an approximate solution whenever $\tilde l_{i-1}\le \mu$, we conclude that NASCAR-Init is guaranteed to terminate at $i=\tilde N_{sc}$, where $\tilde N_{sc} \le \lceil\log_{4} (M_0/\mu) \rceil$.

		\cmo{ In the nonconvex case, a}s described previously in the proof of the NASCAR algorithm, if either ERROR$_i$~=FALSE or $\|\nabla f(\tilde x^i)\|^2>10\tilde l_{i-1}(f(x^0) - f(\tilde x^i))$, then we can confirm that $\tilde l_{i-1}\le l$. In such case, NASCAR-Init will terminate with $l_0 = \tilde l_{i-1}\le l$ and $M_0\le L$. On the other hand, if for all $i$ we always have ERROR$_i$~=FALSE and $\|\nabla f(\tilde x^i)\|^2\le 10\tilde l_{i-1}(f(x^0) - f(\tilde x^i))$, then whenever $\tilde l_{i-1}\le \varepsilon^2/(10(f(x^0) - f(x^*)))$ we have $\|\nabla f(\tilde x^i)\|^2\le \varepsilon^2$. Therefore, we can conclude that NASCAR-Init will terminate either with $\|\nabla f(\hat x)\|\le \varepsilon$, or $l_0\le l$ and $M_0\le 4L$. The former case is guaranteed to happen when $l\le \varepsilon^2/(10(f(x^0) - f(x^*)))$.
		Noting that $\tilde l_{i-1}\le \varepsilon^2/(10(f(x^0) - f(x^*)))$ whenever $i \ge \log_4(40M_0(f(x^0) - f(x^*))/\varepsilon^2)$, we can also conclude that NASCAR-Init is guaranteed to terminate at $i= \tilde N_{nc}$, where $\tilde N_{nc}\le \lceil\max\{1, \log_4(40M_0(f(x^0) - f(x^*))/\varepsilon^2)\}\rceil$. }
		
		\cmo{ Applying Theorem \ref{thm:conv_proposed_unconstrained_sc}, for each iteration $i$, the total number of gradient evaluations required in the two SCAR-PM calls is bounded by $(4+C_1\sqrt{10L/\tilde l_{i-1}})\lceil\log_2(\|\nabla f(x^0)\|/\varepsilon)\rceil + (4+C_1\sqrt{10(L+2\tilde l_{i-1})/\tilde l_{i-1}})\lceil\log_2(4\|\nabla f(x^0)\|/\varepsilon)\rceil$. Observing that $\tilde l_{i-1}\le \tilde l_0=M_0\le L$ and $\tilde l_{i-1} = \tilde l_{0}4^{-i+1}$, the total number of gradient evaluations required by NASCAR-Init after $\tilde N$ iterations is bounded by
			\begin{align}
				\label{eq:tmp}
		\begin{aligned}
			& \sum_{i=1}^{\tilde N}(4+2^{i-1}C_1\sqrt{10L/\tilde l_{0}})\lceil\log_2(\|\nabla f(x^0)\|/\varepsilon)\rceil + (4+2^{i-1}C_1\sqrt{30L/\tilde l_0})\lceil\log_2(4\|\nabla f(x^0)\|/\varepsilon)\rceil
			\\
			\le & 8\tilde N + 2^{\tilde N+1}C_1\sqrt{30L/ M_0}\lceil\log_2(4\|\nabla f(x^0)\|/\varepsilon)\rceil
			.
		\end{aligned}
		\end{align}
	Substituting $\tilde N$ by $\tilde N_{sc}$ and $\tilde N_{nc}$ respectively, we obtain the gradient evaluation complexities for strongly convex and nonconvex cases. 
	
	In the convex case, 
	let $j$ be the integer such that $\tilde l_0\in [2^{j-1}\varepsilon/\operatorname{dist}(x_0,X^*), 2^{j}\varepsilon/\operatorname{dist}(x_0,X^*)]$. Observe that whenever $\tilde l_{i-1} \le 2\varepsilon/\operatorname{dist}(x_0,X^*)$, then the function call SCAR-PM$(f,\varepsilon,x^0,\tilde{l}_{i-1}, M_0)$ will terminate with an approximate solution $\hat x^i$. This is because SCAR-PM$(f,\varepsilon,x^0,\tilde{l}_{i-1}, M_0)$ calls function $(y_1,M_1)=$AR$(f,x^0,\tilde{l}_{i-1}/10, M_0)$ first, while by Theorem \ref{thm:conv_proposed_unconstrained_searchL} such AR function call will always yield an solution $y_1$ with $\|\nabla f(y_1)\|\le \tilde{l}_{i-1}\operatorname{dist}(x_0,X^*)/2$. 
	If $j>1$, note that $\tilde l_{i-1} = 4^{-i+1}M_0 <2\varepsilon/\operatorname{dist}(x_0,X^*)$ whenever $i\ge (j+1)/2$. Observing that $\tilde l_{i-1} = \tilde l_0 4^{-i+1}$, applying Theorem \ref{thm:conv_proposed_unconstrained_sc} and using the alternative bound, the total number of gradient evaluations is bounded by
	\begin{align}
		\begin{aligned}
			&\sum_{i=1}^{\lceil (j+1)/2\rceil}(4+2^{i-1}C_1\sqrt{10L/\tilde l_{0}})\max\{1, j-2i+2\} 
			\\
			& \qquad\qquad + (4+2^{i-1}C_1\sqrt{30L/\tilde l_{0}})\max\{1, j-2i+2\}
			\\
			\le & \sum_{i=1}^{\lceil (j+1)/2\rceil}2^{i+1}C_1\sqrt{30L/\tilde l_{0}}\cdot( j-2i+3)
			\\
			\le & \sum_{i=1}^{\lceil (j+1)/2\rceil}2^{i-(j-1)/2+1}C_1\sqrt{30L\operatorname{dist}(x_0,X^*)/\varepsilon}\cdot( j-2i+3)
			\\
			\le & 64 C_1\sqrt{30L\operatorname{dist}(x_0,X^*)/\varepsilon}.
		\end{aligned}
	\end{align}
	Here in the last inequality we use the fact that 
	\begin{align}
	\begin{aligned}
		& \sum_{i=1}^{\lceil (j+1)/2\rceil}2^{i-(j-1)/2+1}( j-2i+3)
		\\
		= &  \sum_{i=1}^{\lceil (j+1)/2\rceil}2^{ \lceil (j+1)/2\rceil-i+1-(j-1)/2+1}( j-2 \lceil (j+1)/2\rceil+ 2i+1)
		\\
		\le &  \sum_{i=1}^{\lceil (j+1)/2\rceil}2^{ 4-i}\cdot( 2i) \le  32 \sum_{i=1}^{\infty}2^{-i}\cdot i=  64 .
	\end{aligned}
\end{align}	
	
	}
		 
	\qedsymbol\end{proof}
	
	\vgap
	
	From Theorem \ref{thm:NASCAR_conv} and Proposition \ref{pro:NASCAR-init}, when combined with NASCAR-Init, our proposed NASCAR algorithm computes 
	a solution $\hat x$ with $\|\nabla f(\hat x)\|\le\varepsilon$ within at most $\cO(1)\sqrt{Ll}(f(x^0) - f(x^*))/\varepsilon^2 + \cO(1)(\sqrt{L(f(x^0) - f(x^*))}/\varepsilon)\log(\|\nabla f(x^0)\|/\varepsilon)$ gradient evaluations of $\nabla f$. 
	Here the term $\cO(1)\sqrt{Ll}(f(x^0) - f(x^*))/\varepsilon^2 $ is dominant when $l\gg \varepsilon^2/(10(f(x^0) - f(x^*)))$. Moreover, when $l\le \varepsilon^2/(10(f(x^0) - f(x^*)))$, by Proposition \ref{pro:NASCAR-init} NASCAR-Init is already able to compute 
	a solution $\hat x$ with $\|\nabla f(\hat x)\|\le\varepsilon$. The complexity in such case is $\cO(1)(\sqrt{L(f(x^0) - f(x^*))}/\varepsilon)\log(\|\nabla f(x^0)\|/\varepsilon)$, which is nearly optimal (with an extra logarithmic factor) comparing with the lower complexity bound of gradient minimization for convex functions (see \cite{carmon2020lower}). It should be pointed out that NASCAR achieves such complexity without requiring the knowledge of the convexity of $f$. 
	\cmo{Moreover, NASCAR-Init and NASCAR combined together could compute an approximate solution without any knowledge of the problem parameters or the convexity status of the objective function.}

	\section{Concluding remarks}
	\label{sec:conclusion}
	
	In this paper, we propose novel algorithms for minimizing (projected) gradients with optimal complexity. We develop an accumulative regularization strategy and design algorithms that are able to solve approximate solutions with small (projected) gradients with optimal complexities for the convex and strongly convex cases, and update the best known complexity for the nonconvex cases. Our proposed algorithms can be made parameter-free, maintaining these complexity bounds without the need for any problem parameters \cmo{ or the convexity status of the problem. In addition, we also extend our algorithms to constrained and composite optimization problems.}
	
	\cmo{Three} potential question is still open for future research. \cmo{ First, t}he key strategy in our paper is accumulative regularization, which requires us to solve a series of proximal mapping subproblems consecutively. It will be interesting to study whether the algorithm framework can be further simplified, e.g., to a single-loop implementation. Similar questions are also open in the parameter-free implementation: in the strongly convex case, we need to call a series of the AR function to compute an approximate solution. In the nonconvex case, we need to call a series of the SCAR function. It would be interesting to study whether the implementation of parameter-free algorithms can be simplified. 
	\cmo{ 
	Second, the parameters of our algorithms  depend on the accuracy threshold $\varepsilon$ of desired approximate solution. While the value of $\varepsilon$ is necessary in determining the termination of algorithms, it will be interesting to study whether an algorithm could be developed in which the parameters do not depend on the value of $\varepsilon$. 
    Assuming that the target accuracy is unknown, 
    one straightforward strategy is to set the target accuracy to be
    $1/2$ of the current gradient norm when calling our algorithms, and repeat this procedure multiple times until the desired accuracy is reached.
    This strategy will allow us to be more adaptive to the target accuracy without impacting 
    the order of the total gradient complexity. However, it would be an interesting research topic to study whether a better implementation than the aforementioned straightforward strategy can be developed.
    Third, for practical purpose, our proposed algorithm should likely be tuned, for example, to allow more flexible selection of the parameter $\sigma_s$ and $N_s$, early termination, or warm start among different calls to subroutine $\cA$. Due to the length of our paper, we defer some of the above research questions for future studies.
    }


%
%

\bibliographystyle{spmpsci}      
\bibliography{yuyuan}


\end{document}